\documentclass[11pt,reqno]{amsart}
\usepackage{geometry}
\usepackage[english]{babel}
\usepackage{mathdots}
\usepackage{graphicx}
\usepackage{mathtools,amsfonts,amssymb,amsthm,epsfig,epstopdf,url,array,mathrsfs}
\usepackage[utf8]{inputenc}
\usepackage[T1]{fontenc}
\usepackage{caption}
\usepackage{wrapfig}
\usepackage{array}
\usepackage{multirow}
\usepackage{pifont}
\usepackage{xcolor}
\definecolor{PLBgras}{rgb}{0.06,0.42,0.60}
\usepackage[colorlinks=true,linkcolor=PLBgras, citecolor=PLBgras, bookmarkstype=toc]{hyperref}
\hypersetup{urlcolor=blue}
\usepackage{enumitem}
\usepackage{stmaryrd}
\usepackage{bbold}
\usepackage[colored]{shadethm} 
\usepackage{tikz}
\usepackage{pgf}
\usepackage{listings}
\usepackage{longtable}
\usepackage{manfnt}
\usepackage{pgfplots}
\usetikzlibrary{patterns}
\pgfplotsset{compat=1.14}
\usepackage[all]{xy}
\usepackage{epsfig}
\usepackage{cancel}
\usepackage{cases}
\usepackage{xspace}
\usepackage{xparse}
\usepackage{cleveref}

\definecolor{BB}{RGB}{162, 22, 22}
\definecolor{NS}{RGB}{12,133,23}
\definecolor{PLB}{rgb}{0.06,0.42,0.60}
\definecolor{PLBfonce}{rgb}{0.06,0.42,0.60}
\definecolor{PLBmoyen}{RGB}{176, 216, 232}
\definecolor{PLBpale}{rgb}{0.94,0.965,0.965}

\usepackage{framed}

\theoremstyle{definition}
\newtheorem{deff}{Definition}
\newtheorem{assumption}{Assumption}

\theoremstyle{plain}
\newtheorem{prop}[deff]{Proposition}
\newtheorem{thrm}[deff]{Theorem}
\newtheorem{coro}[deff]{Corollary}
\newtheorem{lemma}[deff]{Lemma}
\newtheorem{procedure}[deff]{Method}

\theoremstyle{remark}

\newtheorem{rem}[deff]{Remark}

\definecolor{shadethmcolor}{rgb}{0.95,0.95,0.95}
\definecolor{shaderulecolor}{rgb}{0,0,0}
\newshadetheorem{theoreme}[deff]{Théorème} 
\newshadetheorem{pr}[deff]{Proposition} 
\newshadetheorem{lemme}[deff]{Lemme} 
\newshadetheorem{conj}[deff]{Conjecture}
\newshadetheorem{cor}[deff]{Corollaire} 
\newshadetheorem{defi}{Définition} 
\newshadetheorem{pb}{Problème} 
\newshadetheorem{qu}{Question}

\newcommand{\N}{\mathbb{N}}

\newcommand{\R}{\mathbb{R}}
\newcommand{\D}{\mathbb{D}}

\renewcommand{\S}{\mathbb S}
\newcommand{\Se}{\mathbb{S}_{\varepsilon}}

\newcommand{\C}{\mathbb{C}}

\renewcommand{\>}{\geqslant}
\newcommand{\<}{\leqslant}

\newcommand{\interval}[2]{\llbracket #1:#2 \rrbracket}

\newcommand{\eps}{\varepsilon}

\newcommand{\ind}{\text{Ind}}

\newcommand{\Bbord}{\mathcal B}
\newcommand{\Cconnect}{\mathcal C}
\newcommand{\Pbord}{\mathcal P}
\newcommand{\Uvec}[1]{\widetilde{\mathfrak{U}}_{#1}(z)}
\newcommand{\Uvecmoins}{\mathfrak{U}_-}
\newcommand{\Uvecplus}{\mathfrak{U}_+}

\newcommand{\fun}[5]{#1 : #2 \in #3 \mapsto  #4\in #5}

\newcommand{\egdef}{\overset{def}{=}}

 \newcommand{\DKLindep}{\Delta}
 \newcommand{\DKLindepinvers}{\widetilde{\Delta}}

\newcommand{\PP}{P}
\newcommand{\RR}{R}
\newcommand{\QQ}{Q}
\newcommand{\CC}{C}
\newcommand{\EEs}{\mathbb E_s}
\newcommand{\EEu}{\mathbb E_u}

\newcommand{\U}{\mathcal U}
\newcommand{\courbe}{\DKLindep(\S)}
\newcommand{\Es}{\mathcal E^s(z)}
\NewDocumentCommand{\DKL}{s}{    
\IfBooleanTF{#1}		            
 {\Delta_{0}}{\Delta_{\mathrm{KL}}}}   

\newcommand{\coefpol}{\sigma}

 \renewcommand{\H}{K}
 \newcommand{\HH}{\varphi}

 \newcommand{\mult}{\beta} 
 \newcommand{\nbrmult}{M}

 \newcommand{\borneinfintervalspatial}{x_{\sigma}}

\newcommand{\pe}{\vspace{10pt}}

\newcommand{\symbole}{\gamma}
\newcommand{\dt}{{\Delta t}} 			
\newcommand{\dx}{{\Delta x}}			

\newcommand{\Ind}{\mathrm{Ind}}

\definecolor{codegreen}{rgb}{0,0.6,0}
\definecolor{codegray}{rgb}{0.5,0.5,0.5}
\definecolor{codepurple}{rgb}{0.58,0,0.82}
\definecolor{backcolour}{rgb}{0.95,0.95,0.92}
\lstdefinestyle{mystyle}{
    backgroundcolor=\color{backcolour},   
    commentstyle=\itshape\color{codegreen},
    keywordstyle=\color{PLBfonce},
    numberstyle=\tiny\color{codegray},
    stringstyle=\color{codepurple},
    basicstyle=\ttfamily\footnotesize,
    breakatwhitespace=false,         
    breaklines=true,                 
    captionpos=b,                    
    keepspaces=true,                 
    numbers=left,                    
    numbersep=5pt,                  
    showspaces=false,                
    showstringspaces=false,
    showtabs=true,                  
    tabsize=2
}

\setitemize[0]{label = $\bullet$}

\begin{document}


\title[Stability of finite difference schemes for the hyperbolic IBVP]{Stability of finite difference schemes for the hyperbolic initial boundary value problem by winding number computations}

\author{Benjamin~Boutin \and Pierre~Le~Barbenchon \and Nicolas~Seguin}
\address{Univ Rennes, CNRS, IRMAR - UMR 6625\\F-35000 Rennes, France.}
\email{benjamin.boutin@univ-rennes1.fr}
\email{pierre.lebarbenchon@univ-rennes1.fr}
\address{IMAG, Inria d'Université Côte d'Azur, Univ. Montpellier, CNRS, Montpellier, France}
\email{nicolas.seguin@inria.fr}
\thanks{This work has been partially supported by ANR project NABUCO, ANR-17-CE40-0025 and by Centre Henri Lebesgue, program ANR-11-LABX-0020-0.}
\date{\today}

\begin{abstract}
    In this paper, we present a numerical strategy to check the strong stability (or GKS-stability) of one-step explicit finite difference schemes for the one-dimensional advection equation with an inflow boundary condition. 
    The strong stability is studied using the Kreiss-Lopatinskii theory. We introduce a new tool, the intrinsic Kreiss-Lopatinskii determinant, which possesses the same regularity as the vector bundle of discrete stable solutions. By applying standard results of complex analysis to this determinant, we are able to relate the strong stability of numerical schemes to the computation of a winding number, which is robust and cheap. 
    The study is illustrated with the O3 scheme and the fifth-order Lax-Wendroff (LW5) scheme together with a reconstruction procedure at the boundary.
\end{abstract}

\maketitle

\noindent {\small {\bf AMS classification:} 65M12, 65M06, 30E10}

\noindent {\small {\bf Keywords:} IBVP, Kreiss-Lopatinskii determinant, GKS-stability, finite difference methods, winding number}

\medskip




\makeatletter
\setlength{\parindent}{15\p@} 
\makeatother

\section{Introduction}



\subsection{Motivations and assumptions}

The purpose of this work is to establish an efficient numerical strategy to determine whether a given finite difference method on the half line is stable or not.
We work on an approximation of the rightgoing linear transport equation set on the positive real axis:
\begin{align}\label{eq2:advection}
        \begin{cases}
         \partial_t u + a \partial_x u = 0, &  t\>0, x \>0,\\
        u(t,0) = g(t), & t\>0, \\
         u(0,x) = f(x), & x\>0,
        \end{cases}
    \end{align}
    where $u(t,x)\in\R$ is the unknown, $f$ an initial datum at time $t=0$, $g$ is a prescribed physical boundary datum at the point $x=0$ which corresponds to the inflow boundary because the velocity $a$ is assumed to be positive $a>0$.


\medskip

At the discrete level, we consider explicit one-step finite difference methods of the form
\begin{equation}\label{eq2:stdscheme}
    U_{j}^{n+1} = \displaystyle \sum_{k = -r}^p a_k U_{j+k}^n,
\end{equation}
with integers $r,p\> 1$ and $a_p, a_{-r}$ non zero. The case where $p=0$ or $r=0$ will be discussed in Section~\ref{sec2:casep=0}. Here, the unknown of the scheme $U_j^n$ is expected to approximate the quantity~$u(n\dt, j\dx)$. The time step $\dt>0$ and the space step $\dx>0$ are usually chosen with respect to some CFL condition $\lambda = a\dt/\dx \<\lambda_{\textsf{CFL}}$ discussed later on.

\medskip

Throughout this paper we denote $\S= \{z\in \C, |z|=1\}$ the unit circle, $\D= \{z\in \C, |z|<1\}$ the open unit disk, $\mathcal U= \{z\in \C, |z|>1\}$ the associated exterior domain and $\overline{\mathcal U} = \{z\in \C, |z|\> 1\}$ its closure. For $n<m$, the notation $\interval{n}{m}$ is for the set $\{k\in \N, n\<k\<m\}$.

\medskip

As a central idea in numerical analysis, the Lax equivalence theorem \cite{Lax56} asserts that a consistent scheme is convergent if and only if it is stable. Therefore, all along the paper, only consistent numerical schemes are considered and the discussion concentrates only on their stability issues. The Cauchy-stability for the space-periodic problem is handled with the Fourier symbolic analysis, the so-called Von-Neumann stability analysis (see \cite{Courant28} and \cite{Crank47}) and makes use of the symbol~$\gamma$.
The \emph{symbol} associated with the scheme \eqref{eq2:stdscheme} is defined, for $\xi \in \R$, by
\begin{equation}\label{eq2:symbol}
    \symbole(\xi) =\sum_{k = -r}^p a_k e^{ik\xi}.
\end{equation}

\begin{assumption}\label{assumption2:cauchystab}
    The scheme~\eqref{eq2:stdscheme} is \emph{Cauchy-stable}, meaning that the symbol $\gamma$ satisfies
    $|\symbole(\xi)|\<1$ for all $\xi\in \R$.
    \end{assumption}
When dealing with discrete schemes set over the full line $j\in\mathbb{Z}$, the algebraic characterization of the Cauchy-stability follows classically from the Fourier analysis but in the scalar case, it reduces to a geometric property concerning the \emph{symbol curve} $\Gamma$ which is a closed complex parametrized curve defined by
\[
        \Gamma = \{\theta\in[0,2\pi]\mapsto \gamma(\theta)\}.
    \]
%
%
%
This curve enables a geometric interpretation of the Cauchy-stability assumption~\ref{assumption2:cauchystab} which can be seen as the inclusion $\Gamma\subset\overline{\D}$.
The stability condition~\ref{assumption2:cauchystab} can be easily illustrated graphically in the complex plane. In some sense, our goal is to extend this kind of graphical study when including the numerical boundary conditions.

\medskip

For solving the Initial Boundary Value Problem (IBVP)~\eqref{eq2:advection} with the discrete scheme~\eqref{eq2:stdscheme}, $r$ additional ghost points are needed to take into account the left boundary condition and to fully define the discrete approximation. 
We assume that the values at these ghost points are obtained from a linear combination of the 
first values of the solution close to the boundary and at the same time step. More clearly, the considered numerical schemes reads
\begin{numcases}{} 
    U_{j}^{n+1} = \sum_{k = -r}^p a_k U_{k+j}^n, & $j\in \N,\ n\in \N$, \label{eq2:eqprincip}\\
    U_j^{n} = \sum_{k = 0}^{m-1} b_{j,k} U_k^n + g_j^n,  & $j\in\interval{-r}{-1},\ n\in \N$, \label{eq2:eqbord}\\
    U_j^0 = f_j, & $j\in \N$,\label{eq2:eqinit}
\end{numcases}
where the integer $m$ satisfies $p+r\<m$, $(f_j)_j$ are approximations of the initial condition $f$ and $(g_j^n)_{n,j}$ are numerical data related to the boundary datum $g$ and possibly its derivatives (see for instance the example in Section~\ref{sec2:boundarycondition}). The assumption $p+r\<m$ 
is not restrictive since some of the coefficients $b_{j,k}$ are possibly zero.



In order to define the stability on $\ell^2(\N)$ and for the sake of convenience in the Kreiss-Lopatinskii determinant formulation (see Definition~\ref{def2:detKL}), the explicit use of the $r$ ghost points~$U_j^n$, for $j\in\interval{-r}{-1}$, can be avoided by substituting the $r$ boundary condition~\eqref{eq2:eqbord} into 
the recurrence formula \eqref{eq2:eqprincip} for $j \in \interval{0}{r-1}$. After straightforward calculations, the boundary part reads also under the form  
\begin{equation}\label{eq2:equationBordQuasiToep}\mathfrak U^{n+1}_{r} = \Bbord \mathfrak U^n_{m} + \mathcal G^n\end{equation}
where we denote 
\[\mathfrak U^{n+1}_{r} = \begin{pmatrix} U_{0}^{n+1} \\ \vdots \\ U_{r-1}^{n+1}\end{pmatrix},~\mathfrak U^n_{m} =  \begin{pmatrix} U_{0}^{n} \\ \vdots \\ U_{m-1}^{n}\end{pmatrix},~\mathcal G^n = \begin{pmatrix} 
    a_{-r} & \cdots  &  a_{-1} \\
    & \ddots & \vdots \\
    0 & & a_{-r}
\end{pmatrix} \begin{pmatrix} g_{-r}^n \\ \vdots \\ g_{-1}^n\end{pmatrix}\in \mathcal M_r(\C).\] Here, the matrix $\Bbord\in \mathcal M_{r,m}(\C)$ encodes the boundary treatment in another way. It corresponds to the boundary part of the quasi-Toeplitz matrix form of the scheme used by Beam and Warming \cite{Beam93}. In the detail, the explicit relationship between $\Bbord$ and $B$ is as follows:
\begin{equation}\label{eq2:bordB}
    \Bbord = 
    \begin{pmatrix} 
        a_{-r} & \cdots  &  a_{-1} \\
        & \ddots & \vdots \\
        0 & & a_{-r}
    \end{pmatrix}B + 
    \begin{pmatrix} 
        a_0 & \cdots &  a_p & 0 & \cdots & \cdots &  \cdots &  0 \\
        \vdots &\ddots  &  & \ddots &  \ddots &  & &\vdots  \\
        a_{-r+1} & \cdots & a_0 &\cdots  & a_p &  0 & \cdots & 0 \\ 
    \end{pmatrix} \in \mathcal M_{r,m}(\C)
\end{equation}
with the notation 
$$B =\begin{pmatrix} 
    b_{-r,0} & \cdots &\cdots & b_{-r,m-1} \\
     \vdots  &  & & \vdots \\
    b_{-1,0} & \cdots &\cdots & b_{-1,m-1} \end{pmatrix}\in \mathcal M_{r,m}(\C).$$
%
For example, for the very naive scheme $U_{j}^{n+1}  = \frac{U_{j-1}^n + U_{j+1}^n}{2}$ and the boundary condition $U_{-1}^n = \frac{U_0^n + U_1^n}{2}$, we obtain
$B = \begin{pmatrix} \frac 1 2 & \frac 1 2\end{pmatrix} \text{ and } \Bbord = \begin{pmatrix} \frac 1 4 & \frac 3 4\end{pmatrix}$.

This class of boundary conditions, \eqref{eq2:eqbord} or~\eqref{eq2:equationBordQuasiToep}, encompasses the Dirichlet and Neumann extrapolation procedures, for example, refer to the work of Goldberg~\cite{Goldberg77}. This class also takes into account the more general simplified inverse Lax-Wendroff procedure analyzed by Vilar and Shu~\cite{Vilar15} in the framework of central compact schemes, and Li, Shu and Zhang for the advection equation~\cite{Li16} and for diffusion equations~\cite{Li17}.  
We will focus on the 
so-called reconstruction technique for the boundary condition, which enables to deal with a boundary which is not superposed with a grid point (presented by Dakin, Després and Jaouen~\cite{Dakin18} and also in Section~\ref{sec2:boundarycondition}) in our numerical examples. Other treatments at the boundary  exist, as for example absorbing boundary conditions~\cite{Engquist77} and~\cite{Ehrhardt10}, or transparent boundary conditions~\cite{Arnold03} and~\cite{Coulombel19}, however, in general, they do not enter the present framework.

\medskip



For finite difference schemes applied to discrete IBVP's, the stability study is a principal issue and is the subject of different approaches.
For example, Beam and Warming~\cite{Beam93} study the spectral properties of the Toeplitz or quasi-Toeplitz representation of the scheme.
In the same spirit, the computation of the spectral radius of the truncated (i.e. finite dimensional) quasi-Toeplitz matrix may provide significant information for the power boundedness of the method. This is the method used for example by Dakin, Després and Jaouen~\cite{Dakin18}. This strategy is sometimes called \emph{eigenvalue spectrum visualization method}, especially by Li, Shu and Zhang~\cite{Li22, Li16,Li17}. In Section~\ref{sec2:numerical}, we will compare this latter approach with our own strategy presented hereafter for the O3 scheme in the case of reconstruction boundary conditions.
Our strategy is based on the so-called GKS-stability theory introduced by Gustafsson, Kreiss and Sundström \cite{Gustafsson72} which handles the discrete IBVP \eqref{eq2:eqprincip}-\eqref{eq2:eqbord}-\eqref{eq2:eqinit} with a zero initial data. The reader can  refer to the work by Wu \cite{Wu95} and Coulombel \cite{Coulombel11} for more recent developments on semigroup estimates in order to deduce the stability of the discrete IBVP \eqref{eq2:eqprincip}-\eqref{eq2:eqbord}-\eqref{eq2:eqinit} with non zero initial data from the GKS-stability.
The notion of GKS-stability (or also called strong stability) for the boundary problem makes use of the following discrete norms:

\begin{equation*}
    \|U_j\|_{\dt}^2 = \sum_{n=0}^{+\infty}\dt |U_j^n|^2 \text{\quad and \quad}\|U\|_{\dx,\dt}^2 = \sum_{n=0}^{+\infty} \sum_{j = 0}^{+\infty}\dt\dx |U_j^n|^2.
\end{equation*}
%
The so-called strong stability, or GKS-stability, is defined by:
\begin{deff}[Strong stability]\label{def2:defstabilite}
    The scheme \eqref{eq2:eqprincip}-\eqref{eq2:eqbord}-\eqref{eq2:eqinit} is strongly stable if, for $(f_j)=0$, there exist $C>0$ and $\alpha_0$, such that for all $\alpha>\alpha_0$, for all boundary data $(g^n_j)$, for all $\dx>0$, for all $n\in \N$, the solution satisfies
    \begin{equation}\label{eq2:stability}\sum_{j = -r}^{-1} \|e^{-\alpha n \dt} U_{j}\|_{\dt}^2  + \left (\dfrac{\alpha- \alpha_0}{\alpha\dt +1}\right )\|e^{-\alpha n \dt} U\|_{\dx,\dt}^2 \< C  \sum_{j=-r}^{-1} \|e^{-\alpha n \dt} g_j\|_\dt^2.
    \end{equation}
\end{deff}
%
%
%
%
%
We warn the reader that $\|e^{-\alpha n \dt} U_{j}\|^2_{\dt}$ 
is here an abuse of notation to describe $\sum_{n=0}^{+\infty}\dt e^{-2\alpha n \dt} |U_j^n|^2$ and similarly for $\|e^{-\alpha n \dt} U\|_{\dx,\dt}^2$.
%
%
This stability definition admits a similar but continuous form for the solutions to continuous hyperbolic PDE’s \cite{Benzoni06}. Namely, it provides some a priori estimates that are useful for a general analysis of such problems.

\medskip


The following Kreiss theorem \cite{Kreiss68} expresses a necessary and sufficient condition for the strong stability.
We provide hereafter a condensed formulation of this theorem, obtained from \cite[Thm~5.1]{Gustafsson72} combined with \cite[Lem~13.1.4]{Gustafsson13} or with \cite[Def~2.23]{Gustafsson08}.

\begin{thrm}[Kreiss]\label{th2:thrmKreiss}
    The following statements are equivalent:
    \begin{enumerate}[label = (\roman*)]
        \item The scheme \eqref{eq2:eqprincip}-\eqref{eq2:eqbord}-\eqref{eq2:eqinit} is strongly stable in the sense of Definition~\ref{def2:defstabilite}.
        \item The Uniform Kreiss-Lopatinskii Condition is satisfied.
    \end{enumerate}
\end{thrm}
The Uniform Kreiss-Lopatinskii Condition corresponds to the absence of zeros for the so-called Kreiss-Lopatinskii determinant $\DKL$ that we present here by this informal definition:
\begin{equation}\label{eq2:definformelDKL}
    \DKL(z) = \det (\mathfrak B e_1(z), \dots, \mathfrak B e_r(z))
\end{equation}
where $(e_1(z), \dots, e_r(z))$ is an explicit basis of the linear space of the $\ell^2(\N)$-stable solutions of the $\mathcal Z$-transform of the interior equation \eqref{eq2:eqprincip} and $\mathfrak B$ is an encoding of the $\mathcal Z$-transform of the boundary equation \eqref{eq2:eqbord}. 
For a proper definition of this determinant, the reader can look at Definition~\ref{def2:detKL} or the book by Gustafsson, Kreiss and Oliger~\cite{Gustafsson13}. 
Before going on, let us provide some comments to a particular case we already studied.


\subsection{The case of totally upwind schemes and summary of \cite{Boutin22}}\label{sec2:casep=0}

The present article is a non trivial extension of our previous work \cite{Boutin22} that deals with the restricted case of totally upwind schemes.
Totally upwind schemes are schemes of the form \eqref{eq2:stdscheme} with $p=0$ if $a>0$ or $r=0$ if $a<0$. Without loss of generality, we restrict here the discussion to the case $p=0$ since flipping the indices may turn a case to the other. 
In this section, we summarize the result of \cite{Boutin22} and introduce the novelty of the present work.
The first step of the analysis conducted in \cite{Boutin22} is based on the introduction of the intrinsic Kreiss-Lopatinskii determinant:
\begin{equation}\label{eq2:defDKLindep}\DKLindep(z) = \dfrac{\det (\mathfrak B e_1(z), \dots, \mathfrak B e_r(z))}{\det (e_1(z),\dots, e_r(z))}\end{equation}
using the same informal notation as in \eqref{eq2:definformelDKL}. 
Under appropriate assumptions, an explicit formula for the intrinsic Kreiss-Lopatinskii determinant is obtained:
\begin{equation}\label{eq2:explicitformula}\forall |z|\>1,\quad \DKLindep(z) = (-1)^{r(m-r)}\det C(z)\left (\dfrac{a_{-r}}{a_0 - z}\right )^{m-r}\end{equation}
where $\det C(z)$ is an computable polynomial in $z$ depending only on the coefficients $(a_j)_{j=-r}^0$ and on~$\mathfrak B$. Thanks to this result, we prove that $\DKLindep$ is holomorphic on $\overline{\U}$. Note that this property may be wrong as long as the standard Kreiss-Lopatinskii determinant is concerned.

Applying the residue theorem to $\DKLindep$, we develop a numerical strategy to count the number of zeros of the Kreiss-Lopatinskii determinant in $\U$. By Theorem~\ref{th2:thrmKreiss} (Kreiss), we conclude that if  $\Ind_{\DKLindep(\S)}(0) < r$ then the scheme is not stable
where $\Ind_{\DKLindep(\S)}(0)$ is the notation for the winding number of $0$ with respect to the Kreiss-Lopatinskii curve $\DKLindep(\S)$.
This result allows us to establish an efficient and practical method (see Method 19 of \cite{Boutin22} or Method~\ref{proc2:numericalprocedure} of the present paper) to study the stability of a scheme with boundary. It provides sharp results for the solution $(U_j^n)_{n\in \N} \in \ell^2(\N)$ to the problem set on the half line $\N$. In particular and contrary to numerical investigations of stability which are based on the computation of the spectral radius, no arbitrary truncation of (quasi-)Toeplitz matrices is needed. In return, a problem set on a bounded space domain needs, for a whole convergence study, superposition techniques for truncated data, as used in~\cite{Coulombel19} and~\cite{Boutin19}. This feature restricts mainly the study to explicit scheme.
\medskip

In the present article, we extend to $p\>1$ the connection between the winding number of $0$ and the stability of the scheme. Indeed, even if we do not have an explicit formula as~\eqref{eq2:explicitformula}, the holomorphic property of the intrinsic Kreiss-Lopatinskii determinant is sufficient to have the same efficient method to study the stability. 
This approach, using the winding number, is robust since instead of finding zeros of an algebraic curve, it only requires the computation of a winding number, which is an integer, to count the number of zeros. 
One can mention the work of Thuné~\cite{Thune86} who develops a numerical method to check the GKS-stability. He looks for the precise location of the zeros of the Kreiss-Lopatinskii determinant approximating the roots of some parameterized characteristic polynomial equations which is significantly different with our work.


\subsection{Outline of the paper}
After constructing the intrinsic Kreiss-Lopatinskii determinant in Section~\ref{sec2:detKL} and Section~\ref{sec2:proof}, we see that, in such a general case, the lack of an explicit formula for $\DKLindep(z)$ does not preclude holomorphic properties (see Theorem~\ref{thrm2:mainholocont}). From there, we obtain the following stability criterion: if $\Ind_{\DKLindep(\S)}(0) < r$ then the scheme is not stable (see Corollary~\ref{thrm2:nbrzerodet}). We prove these results by the use of Hermite interpolation and residue theorem. 
To compute the Kreiss-Lopatinskii determinant numerically, in Section~\ref{sec2:numerical}, we use a an easy-to-use formulation of it which is, in some sense, close to the explicit formulation~\eqref{eq2:explicitformula}. Moreover, Section~\ref{sec2:numerical} gathers the numerical procedure to draw the Kreiss-Lopatinskii curve, several examples, and numerical experiments for illustrating the efficiency of the proposed strategy.

\section{Kreiss-Lopatinskii determinants}\label{sec2:detKL}

In this section, we introduce the Kreiss-Lopatinskii determinant, a usual tool to check the Uniform Kreiss-Lopatinskii Condition. Then we define the intrinsic Kreiss-Lopatinskii determinant, namely a reshaping of the previous one, which is more convenient in practice and has better properties than the classical Kreiss-Lopatinskii determinant: holomorphicity, continuity, independence on the basis\dots


\subsection{Stable subspace $\Es$ and matrix representation}

First, we study the solutions to the interior equation:
\begin{equation}\label{eq2:eqprincip2}
U_{j}^{n+1} = \sum_{k = -r}^p a_k U_{k+j}^n,\  j\in \N,\ n\in \N.
\end{equation}
To study this equation, the $\mathcal Z$-transform (see \cite[Lesson 40]{Gasquet13}) is applied. This transformation is defined for $(x_n)_{n\in \N}\in\ell^2(\N)$ such that $x_0 = 0$ and $z\in\U$ by $\widetilde{x}(z) = \sum_{n\>0} z^{-n}x_n$. The previous equation then reads
\begin{equation}\label{eq2:Ztransform}
    z\widetilde{U}_j(z) = \sum_{k = -r}^p a_k\widetilde{U}_{j+k}(z),\ j\in\N,\ z\in\U.
\end{equation}
%
%
To solve the linear recurrence equation~\eqref{eq2:Ztransform}, let us introduce the following characteristic equation where $z$ plays the role of a parameter and $\kappa$ is the indeterminate:
\begin{equation}\label{eq2:eqcharac}
    z\kappa^r = \sum_{k = -r}^p a_k \kappa^{r+k}.
\end{equation}
This equation is nothing but the discrete dispersion relation of the finite difference scheme~\eqref{eq2:eqprincip2}, with frequency parameter $\kappa$ in space and $z$ in time. It is formally obtained 
by looking for solutions to the interior equation \eqref{eq2:eqprincip2} having the form $U_j^n = z^n \kappa^j$.

\medskip

In the spirit of a classic result by Hersh~\cite{Hersh63}, the following lemma provides a property of separation for the roots with respect to the unit circle.
\begin{lemma}[Hersh]\label{lem2:hersh}
    Assume \ref{assumption2:cauchystab}. For $z$ in the unbounded connected component of $\C\setminus \Gamma$,
    \begin{enumerate}
        \item  there is no root of the characteristic equation \eqref{eq2:eqcharac} on $\S$,
        \item there are $r$ roots (with multiplicity) of the characteristic equation \eqref{eq2:eqcharac} in $\D$ and $p$ roots (with multiplicity) of the characteristic equation \eqref{eq2:eqcharac} in $\U$.
    \end{enumerate}
\end{lemma}

\begin{rem} Under the Cauchy-stability assumption~\ref{assumption2:cauchystab}, the inclusion $\Gamma\subset \overline{\D}$ is known. From there, it follows that the unbounded connected component of $\C\setminus\Gamma$ contains the whole set~$\U$ so that a weaker form of the lemma is available for considering $z \in \U$ only.
If in addition, the considered scheme is also \emph{dissipative}, meaning that its symbol $\symbole$ satisfies
\[|\symbole(\xi)|\< 1 - \delta |\xi|^{2s},\quad \xi \in [-\pi,\pi],\]
for some $\delta>0$ and an integer $s\in \N^*$ independent of $\xi$,
then the same separation result is available for $z\in\overline{\U}\setminus\{1\}$. The reason for that property is that in that case one has $\S\cap\Gamma=\{1\}$.
\end{rem}

\begin{proof}[Proof of Lemma~\ref{lem2:hersh}]\ \par
        \begin{enumerate}
        \item Assume there exists a root $\kappa$ of \eqref{eq2:eqcharac} on the unit circle, then one can find $\theta \in \R$ such that $\kappa = e^{i\theta}$. So we have
    \[
        z = \sum_{j = -r}^p a_j \kappa^j = \sum_{j=-r}^p a_j e^{ij\theta}  = \symbole(\theta).
    \]
    This is a contradiction because $z \in \Gamma$ and by assumption $z \in \C\setminus \Gamma$. This concludes the proof.
        \item We denote $\Cconnect$ the unbounded connected component of $\C \setminus \Gamma$. The polynomial \eqref{eq2:eqcharac} has $p+r$ roots (with multiplicity). It is sufficient to count how many roots there are inside the unit disk to deduce the number of roots outside.
        By continuity of the roots with respect to coefficients and because there is no root on the unit circle for $z\in \Cconnect$, we know that there is a constant number of roots inside the unit disk for all $z\in \Cconnect$.
        By Rouché's theorem, one can study the zeros of $f_z(\kappa) = \kappa^r - \frac{1}{z}(a_{-r} + a_{-r+1} \kappa + \cdots + a_p \kappa^{p+r})$ and $g_z(\kappa) = \kappa^r - \frac{1}{z}a_{-r}$ in $\D$ for $z$ sufficiently large to have the result.
    \end{enumerate}
\end{proof}

Lemma~\ref{lem2:hersh} (Hersh) 
above is illustrated in Figure~\ref{fig2:hersh}.
The first two lines correspond to the Lemma~\ref{lem2:hersh} (Hersh) and the third one describes the possible configuration for $z\in\Gamma\cap\S$, typically not meeting the assumptions. This case will be the object of a subsequent discussion.

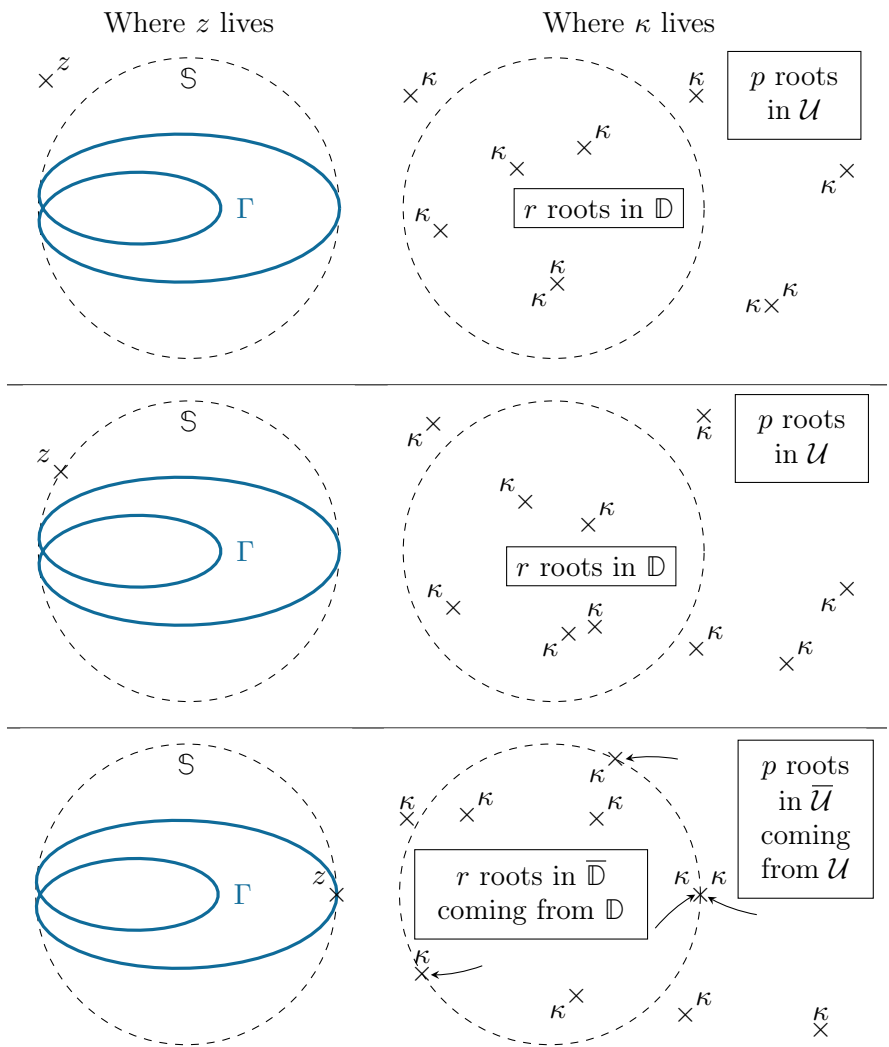
\begin{figure}
    \centering
    \begin{tabular}{cc}
        Where $z$ lives & Where $\kappa$ lives \\
\begin{tikzpicture}
    \draw[white, fill] (-2.2,-2.2) rectangle (2.2,2.2);
    \draw[dashed] (0,0) circle (2);
    \draw[very thick, PLB] plot[domain=0:360,samples=100] ({2.193*(0.72*cos(2*\x) +0.36*cos(\x) - 0.08)-0.188} , {2*(-0.36*sin(2*\x) - 0.18*sin(\x))} );
    \draw  (-1.9,1.7) node{$\times$} node[above right]{$z$};
    \draw  (0.5,0) node[right,PLB]{$\Gamma$};
    \draw  (0,2) node[below]{$\S$};
\end{tikzpicture}

& \begin{tikzpicture}
    \draw[white, fill] (-2.2,-2.2) rectangle (4.2,2.2);
    \draw[dashed] (0,0) circle (2);
    \draw  (-0.4863821504576905,0.5256329266344337) node{$\times$} node[above left]{$\kappa$};
    \draw  (0.4077563762011518,0.7969276234786205) node{$\times$} node[above right]{$\kappa$};
    \draw  (-1.5,-0.3) node{$\times$} node[above left]{$\kappa$};
    \draw  (0.05,-1) node{$\times$} node[below left]{$\kappa$}
    node[above]{$\kappa$};
    \draw  (0.6,0) node[draw]{$r$ roots in $\D$} ;
    \draw  (2.9,-1.3) node{$\times$} node[left]{$\kappa$}
    node[above right]{$\kappa$};
    \draw  (-1.9,1.5) node{$\times$} node[above right]{$\kappa$};
    \draw  (1.9,1.5) node{$\times$} node[above]{$\kappa$};
    \draw  (3.9,0.5) node{$\times$} node[below left]{$\kappa$};
    \draw  (3.2,1.5) node[draw]{\begin{tabular}{c}$p$ roots \\ in $\U$\end{tabular}} ;
\end{tikzpicture}
\\ \hline
\begin{tikzpicture}
    \draw[white, fill] (-2.2,-2.2) rectangle (2.2,2.2);
    \draw[dashed] (0,0) circle (2);
    \draw[very thick, PLB] plot[domain=0:360,samples=100] ({2.193*(0.72*cos(2*\x) +0.36*cos(\x) - 0.08)-0.188} , {2*(-0.36*sin(2*\x) - 0.18*sin(\x))} );
    \draw  (-1.7,1.0535) node{$\times$} node[above left]{$z$};
    \draw  (0.5,0) node[right,PLB]{$\Gamma$};
    \draw  (0,2) node[below]{$\S$};
    
\end{tikzpicture}
&
\begin{tikzpicture}
    \draw[white, fill] (-2.2,-2.2) rectangle (4.2,2.2);
    \draw[dashed] (0,0) circle (2);
    \draw  (-0.3737011056381775,0.6646464581885967) node{$\times$} node[above left]{$\kappa$};
    \draw  (0.46050826642164544,0.3508791988777962) node{$\times$} node[above right]{$\kappa$};
    \draw  (-1.3332841067875879,-0.7465188753809819) node{$\times$} node[above left]{$\kappa$};
    \draw  (0.2,-1.1) node{$\times$} node[below left]{$\kappa$};
    \draw  (0.55,-1) node{$\times$} node[above]{$\kappa$};
    \draw  (0.5,-0.2) node[draw]{$r$ roots in $\D$} ;
    \draw  (1.9,-1.3) node{$\times$} node[above right]{$\kappa$};
    \draw  (3.1,-1.5) node{$\times$} node[above right]{$\kappa$};
    \draw  (-1.6,1.7) node{$\times$} node[below left]{$\kappa$};
    \draw  (2,1.8) node{$\times$} node[below]{$\kappa$};
    \draw  (3.9,-0.5) node{$\times$} node[below left]{$\kappa$};
    \draw  (3.3,1.5) node[draw]{\begin{tabular}{c}$p$ roots \\ in $\U$\end{tabular}} ;
\end{tikzpicture}
\\ \hline
\begin{tikzpicture}
    \draw[white, fill] (-2.2,-2.2) rectangle (2.2,2.2);
    \draw[dashed] (0,0) circle (2);
    \draw[very thick, PLB] plot[domain=0:360,samples=100] ({2.193*(0.72*cos(2*\x) +0.36*cos(\x) - 0.08)-0.188} , {2*(-0.36*sin(2*\x) - 0.18*sin(\x))} );
    \draw  (2,0) node{$\times$} node[above left]{$z$};
    \draw  (0.5,0) node[right,PLB]{$\Gamma$};
    \draw  (0,2) node[below]{$\S$};
\end{tikzpicture}
&
\begin{tikzpicture}
    \draw[white, fill] (-2.2,-2.2) rectangle (4.2,2.2);
    \draw[dashed] (0,0) circle (2);
    \draw  (0.8718,1.8) node{$\times$} node[below left]{$\kappa$};
    \draw  (0.6221277274892463,1) node{$\times$} node[above right]{$\kappa$};
    \draw  (-1.7,-1.0535) node{$\times$} node[above]{$\kappa$};
    \draw[->,>=stealth] (-0.9,-0.95) arc (290:270:2) ;

    \draw  (-1.1,1.0535) node{$\times$} node[above right]{$\kappa$};

    \draw  (2,0) node{$\times$} node[above left]{$\kappa$} node[above right]{$\kappa$};
    \draw[->,>=stealth] (2.75,-0.27) arc (260:240:2) ;
    \draw[->,>=stealth] (1.4,-0.45) arc (140:120:2) ;

    \draw  (3.6,-1.8) node{$\times$} node[above]{$\kappa$};
    \draw  (-1.9,1) node{$\times$} node[above]{$\kappa$};
    \draw  (1.8,-1.6) node{$\times$} node[above right]{$\kappa$};

    \draw  (0.35,-1.34360252969993665) node{$\times$} node[below left]{$\kappa$};
    \draw[->,>=stealth] (1.7,1.8) arc (80:100:2) ;
    \draw  (-0.25,0) node[draw]{\begin{tabular}{c}$r$ roots in $\overline{\D}$\\coming from $\D$\end{tabular}} ;
    \draw  (3.4,1) node[draw]{\begin{tabular}{c}$p$ roots\\in $\overline{\U}$\\coming \\from $\U$\end{tabular}} ;
\end{tikzpicture}
\\
\end{tabular}
\caption{Illustration of Lemma \ref{lem2:hersh}: case $|z|>1$ (first line), case $|z|=1$ and $z \notin \Gamma$ (second line) and case $z \in \Gamma$ where Lemma \ref{lem2:hersh} does not hold (third line).}\label{fig2:hersh}
\end{figure}

For $|z|>1$, by Lemma~\ref{lem2:hersh} (Hersh), the linear subspace of solutions to \eqref{eq2:Ztransform} living in $\ell^2(\N)$ is generated by the following $r$ vectors:



\begin{equation}\label{eq2:basis}
    \begin{pmatrix}1 \\ \kappa_\ell  \\\kappa_\ell ^2 \\ \kappa_\ell ^3 \\\kappa_\ell ^4 \\ \vdots \end{pmatrix}, \begin{pmatrix}0 \\ \kappa_\ell  \\ 2\kappa_\ell ^2 \\ 3\kappa_\ell ^3 \\ 4\kappa_\ell ^4 \\ \vdots \end{pmatrix},\begin{pmatrix}0 \\ \kappa_\ell  \\ 2^{2}\kappa_\ell ^2 \\ 3^{2}\kappa_\ell ^3 \\ 4^{2} \kappa_\ell ^4 \\ \vdots \end{pmatrix},  \dots,\begin{pmatrix}0 \\ \kappa_\ell  \\ 2^{\mult_\ell -1}\kappa_\ell ^2 \\ 3^{\mult_\ell -1}\kappa_\ell ^3  \\ 4^{\mult_\ell -1}\kappa_\ell ^4\\ \vdots \end{pmatrix},\quad  \ell = 1,\dots, \nbrmult\end{equation}
where $\kappa_1, \dots, \kappa_\nbrmult$ of multiplicity $\mult_1, \dots, \mult_\nbrmult$ are the solutions to \eqref{eq2:eqcharac} living in $\D$, with ${\mult_1 +\cdots + \mult_\nbrmult = r}$ 
(we omit the $z$-dependence of $\kappa(z)$ for the sake of readability).

\medskip

\noindent\textbf{Notation.} We denote $\Es$ the linear subspace of solutions to \eqref{eq2:Ztransform} living in $\ell^2(\N)$ and $\H_{i,j}(z) \in \mathcal M_{j-i+1,r}(\C)$ the matrix where we put in columns the extraction of all the lines between $i$ and $j$ (included) of the $r$ vectors of~\eqref{eq2:basis}, where $0\< i \<j$.

\begin{rem}\label{ex2:kappadistinct}
    For $r=2$, if the solutions to \eqref{eq2:eqcharac} are $\kappa_1(z) \neq \kappa_2(z)$, then there are exactly two roots with multiplicity~1. The solutions to \eqref{eq2:Ztransform} can be written
    $\widetilde{U}_j(z) = \alpha_1\kappa_1(z)^j + \alpha_2 \kappa_2(z)^j,$
    and we have
\[
    \H_{0, 2}(z) = \begin{pmatrix}1 & 1 \\ \kappa_1(z) & \kappa_2(z)\\\kappa_1(z)^2& \kappa_2(z)^2  \end{pmatrix}.\]
\end{rem}

\begin{rem}\label{ex2:kapparacinedouble}
    Still for $r=2$, if the solution to \eqref{eq2:eqcharac} now is $\kappa(z)$ with multiplicity 2, then the solutions to \eqref{eq2:Ztransform} can be written  $\widetilde{U}_j(z) = (\alpha_1+\alpha_2 j)\kappa(z)^j$, and we have
\[\H_{0,3}(z) = \begin{pmatrix} 1 & 0\\ \kappa(z) & \kappa(z)\\ \kappa(z)^2 & 2\kappa(z)^2 \\ \kappa(z)^3 & 3 \kappa(z)^3 \end{pmatrix}.\]
\end{rem}

We raise awareness of the dependence in $z$ and of the continuity issues because the map $z \mapsto \H_{i,j}(z)$ is not continuous whereas the set of roots of \eqref{eq2:eqcharac} is a continuous mapping with respect to $z$. Indeed, the root curves $(\kappa_j(z))_{j}$ can intersect, when a multiple root occurs. For example, for $r=2$, if there is $(z_n)_{n\in \N} \subset \U$ with $\kappa_1(z_n) \neq \kappa_2(z_n)$ which converge to $z_{\infty}\in \U$ such that $\kappa_1(z_{\infty}) = \kappa_2(z_{\infty})$ a double root, then we have 
$$\forall j \in \{1,2\}, \quad \kappa_j(z_n) \xrightarrow[n\to \infty]{} \kappa_j(z_{\infty})$$ but $$\H_{0,3}(z_n) = \begin{pmatrix} 1 & 1 \\ \kappa_1(z_n) & \kappa_2(z_n) \\ \kappa_1^2(z_n) & \kappa_2^2(z_n) \\ \kappa_1^3(z_n) & \kappa_2^3(z_n) \end{pmatrix} \cancel{\xrightarrow[n\to \infty]{}} ~\H_{0,3}(z_{\infty}) = \begin{pmatrix} 1 & 0\\ \kappa_1(z_{\infty}) & \kappa_1(z_{\infty}) \\ \kappa_1^2(z_{\infty}) & 2 \kappa_1^2(z_{\infty}) \\ \kappa_1^3(z_{\infty}) & 3 \kappa_1^3(z_{\infty}) \end{pmatrix}.$$
Consequently, the considered basis \eqref{eq2:basis} of $\Es$ does not generally define a continuous mapping with respect to $z$. 

\medskip

In spite of the difficulty enlightened above, it turns out that $\Es$ is a continuous and even holomorphic vector bundle over $\U$ as it is discussed in \cite[Thm 4.3]{Coulombel11}. It is also proved that this vector bundle $\Es$ can even be continuously extended over $\overline{\U}$, thus considering $z\in\S$ as well (see also \cite{Metivier04} for a similar property for the hyperbolic-parabolic PDE case).
The main point therein is that for some $z_0 \in \S$, there may exists one (or several) root $\kappa_0(z_0)$ of \eqref{eq2:eqcharac} on $\S$. At such points $z_0$ the Lemma~\ref{lem2:hersh} (Hersh) does not hold anymore. This situation is depicted on the third line of Figure~\ref{fig2:hersh}.
For $z$ on $\S$, the space $\Es$ still is of dimension $r$ and we extend the notation $\H_{i,j}(z)$. 
%
%
%
%
%
%
We can summarize the above discussion in the following theorem.

\begin{thrm}[\cite{Coulombel11}] \label{th2:CoulombelHolo}
    Under assumption \ref{assumption2:cauchystab}, the space $\Es$ is a holomorphic vector bundle over $\U$ and can be extended in a unique way to a continuous vector bundle over $\overline{\U}$.
\end{thrm}

\begin{rem}\label{rem2:rootsfrominsideoutside}For the extension, the first difficulty is to select the roots of \eqref{eq2:eqcharac} coming from the inside, indeed, if there is a root on $\S$, it can be coming from the inside of $\D$, the outside or both (in case of multiplicity). In Section~\ref{sec2:numDeltaselection}, we will explain the numerical strategy to select the good ones.
The second difficulty is to prove the continuity of $\Es$ after the extension, it follows from the existence of a K-symmetrizer and is obtained e.g. in \cite[Thm 4.3]{Coulombel11}. As previously observed,  $\H_{i,j}(z)$ is generally not continuous with respect to $z$.
\end{rem}


\subsection{Intrinsic Kreiss-Lopatinskii determinant}
In this section, we define properly formulas \eqref{eq2:definformelDKL} and \eqref{eq2:defDKLindep}. 
Let us consider the $\mathcal Z$-transformed version of the boundary condition~\eqref{eq2:equationBordQuasiToep}, that is
\begin{equation}\label{eqbordZtransform}
    z \begin{pmatrix}\widetilde{U}_0(z) \\ \vdots \\  \widetilde{U}_{r-1}(z) \end{pmatrix}  - \Bbord \begin{pmatrix}\widetilde{U}_0(z) \\ \vdots \\  \widetilde{U}_{m-1}(z) \end{pmatrix} =  \begin{pmatrix} 
        a_{-r} & \cdots  &  a_{-1} \\
        & \ddots & \vdots \\
        0 & & a_{-r}
    \end{pmatrix} \begin{pmatrix}\widetilde{g}_{-r}(z) \\ \vdots \\ \widetilde{g}_{-1}(z) \end{pmatrix}.
\end{equation}
%
%
%
Injecting the solution $(\widetilde{U}_j(z))_{j\in \N}\in \Es$ to~\eqref{eq2:Ztransform} into \eqref{eqbordZtransform}, we obtain a system of $r$ equations with~$r$ scalar unknowns: they are the coefficients of $(\widetilde{U}_j(z))_{j\in \N}$ written in the basis~\eqref{eq2:basis} of $\Es$. 

\begin{rem} For $r=2$ and a given value of $z$ (we skip for convenience the dependence in $z$ hereafter), if $\kappa_1 \neq \kappa_2$ so that the solution to~\eqref{eqbordZtransform} has the form $\alpha_1\kappa_1^j + \alpha_2 \kappa_2^j$, then that solution is constrained by the system~\eqref{eqbordZtransform}.
The matricial form of that system reads
\[
    \left (z\begin{pmatrix} 1 & 1 \\ \kappa_1 & \kappa_2 \end{pmatrix}  - \Bbord \begin{pmatrix}  1 & 1 \\ \kappa_1 & \kappa_2 \\ \kappa_1^2 & \kappa_2^2 \\  \vdots & \vdots  \\ \kappa_1^{m-1} & \kappa_2^{m-1}   \end{pmatrix}  \right ) \begin{pmatrix} \alpha_1 \\ \alpha_2 \end{pmatrix} = \begin{pmatrix} a_{-1}\widetilde{g}_{-1} + a_{-2}\widetilde{g}_{-2} \\ a_{-2}\widetilde{g}_{-1}
    \end{pmatrix}.
\]
The injectivity, whence invertibility, of the boundary condition is thus directly related to the property $\det (z\H_{0,1} - \Bbord \H_{0,m-1}(z))\neq 0$, where $z\H_{0,1} - \Bbord \H_{0,m-1}(z)\in \mathcal M_{2,2}(\C)$.
\end{rem}

\begin{deff}[Kreiss-Lopatinskii determinant]\label{def2:detKL}
    The \emph{Kreiss-Lopatinskii determinant} is the complex-valued function defined for $|z|\>1$ by:
    \[\DKL(z) = \det(z\H_{0,r-1}(z) - \Bbord \H_{0,m-1}(z)).\]
\end{deff}

Despite the fact that the space $\Es$ is a holomorphic vector bundle over $\U$ and continuous over~$\overline{\U}$ (Theorem~\ref{th2:CoulombelHolo}), this determinant $\DKL$ is not holomorphic on $\U$. To retrieve those properties, we define the \emph{intrinsic Kreiss-Lopatinskii determinant} $\DKLindep$ that we can motivate by the following informal discussion.
The above Kreiss-Lopatinskii determinant is actually not well defined until we order in some way the roots $(\kappa_j(z))_{j=1,\dots,r}$ of~\eqref{eq2:eqcharac}. There are two points to emphasize. The first one is related to crossing roots and already discussed after Remark~\ref{ex2:kapparacinedouble}. The second one is that, outside crossing cases, being given any choice for the ordering of the roots (and thus of the vectors of the basis~\eqref{eq2:basis} for the vector bundle), there is in general no chance to obtain a holomorphicity property for the components of the matrix $\H_{0,m-1}(z)$ over $\U$. For example, even the roots of $X^2-z$ are not holomorphic w.r.t $z\in \U$ because of the logarithm determination. On the other side, any symmetric functions
of the roots $(\kappa_j(z))_{j=1,\dots,r}$ however are holomorphic because they can be obtained directly in terms of the coefficients of the polynomial~\eqref{eq2:eqcharac}. So, except for crossing roots, the same holds for the quantity $\DKL(z)$ since the matrices $B$ and $\Bbord$ are constants and the determinant itself is a symmetric function.

A very natural way to reach the holomorphic property and go beyond the last difficulties consists in dividing $\DKL$ by the quantity $\det \H_{0,r-1}(z)$. Hence, the same permutation or combination of the vectors of the basis~\eqref{eq2:basis} is
involved in both computations. This \emph{intrinsic Kreiss-Lopatinskii determinant} has already been introduced and studied in a particular case in~\cite{Boutin22}.

\begin{deff}[Intrinsic Kreiss-Lopatinskii determinant]
    The \emph{intrinsic Kreiss-Lopatinskii determinant} is the complex-valued function defined for $|z|\>1$ by:
\begin{equation}\label{eq2:dklintrinsic}
    \DKLindep(z)=\dfrac{\DKL(z)}{\det \H_{0,r-1}(z)}.
\end{equation}
\end{deff}
Let us note that the intrinsic Kreiss-Lopatinskii determinant can be rewritten
\begin{equation}\label{eq2:DKLavecKK}
    \DKLindep(z) = \dfrac{\det(z\H_{0,r-1}(z) - \Bbord \H_{0,m-1}(z))}{\det \H_{0,r-1}(z)}  = 
    z^r \det \left ( I_r - \dfrac{\Bbord \H_{0,m-1}(z) \H_{0,r-1}(z)^{-1}}{z}\right ).
\end{equation}
To conclude with these definitions, let us state a little more about the \emph{Uniform Kreiss-Lopatinksii Condition}. With the above notations and additionally to the invertibility of $z\H_{0,r-1}(z) - \Bbord \H_{0,m-1}(z)$, it corresponds to the existence of a constant $C>0$ such that for any $z\in\overline{\U}$, any $\widetilde{U}\in\Es$ solution to~\eqref{eqbordZtransform} satisfies the uniform estimate $\|\widetilde{U}\| \< C \|\widetilde{g}\|$.
From the Parseval identity for the $\mathcal{Z}$-transform, this inequality gives directly the first necessary half-part of the strong stability estimate~\eqref{eq2:stability}. We refer the reader to~\cite{Gustafsson13} for a more detailed presentation.


\subsection{Main results} \label{sec2:result}
Theorem~\ref{thrm2:mainholocont} is our main theoretical result. 
It states that the intrinsic Kreiss-Lopatinskii determinant has the same regularity properties as $\Es$, see Theorem~\ref{th2:CoulombelHolo}.


\begin{thrm}[Smoothness of the intrinsic Kreiss-Lopatinskii determinant]\label{thrm2:mainholocont}
    Assume \ref{assumption2:cauchystab}. The intrinsic Kreiss-Lopatinskii determinant $\DKLindep$ is holomorphic on $\U$ and continuous on $\overline{\U}$.
\end{thrm}

By equation~\eqref{eq2:dklintrinsic}, the function $\DKLindep$ shares the same zeros with the Kreiss-Lopatinskii determinant~$\DKL$, so that it can be used as an alternative in the Uniform Kreiss-Lopatinskii Condition, see Theorem~\ref{th2:thrmKreiss} (Kreiss).
Another property, important for the forthcoming applications, lies in the next Corollary~\ref{thrm2:nbrzerodet} and involves the following important geometrical object:
\begin{deff}
    The \emph{Kreiss-Lopatinskii curve} $\DKLindep(\S)$ is the closed complex parameterized curve
    \[ \DKLindep(\S) = \{ \theta\in[0,2\pi]\mapsto \DKLindep(e^{i\theta})\}.\]
\end{deff}

Using the residue theorem\footnote{All the complex analysis results can be found in \cite{Lang13}.} thanks to  Theorem \ref{thrm2:mainholocont}, we obtain the following result.
\begin{coro}[Number of zeros of the intrinsic Kreiss-Lopatinskii determinant]\label{thrm2:nbrzerodet}
    Assume \ref{assumption2:cauchystab}. If $0 \notin \DKLindep(\S)$ then the equation $\DKLindep(z) = 0$ has exactly $r - \Ind_{\DKLindep(\S)}(0)$ zeros in $\U$.
\end{coro}

Here above and in all the paper, $\Ind_{\DKLindep(\S)}(0)$ denotes the winding number of the origin with respect to the closed oriented curve $\DKLindep(\S)$ (see \cite{Lang13} for a definition of the winding number). This corollary helps us to establish an efficient and practical method to study the stability of a given IBVP through Theorem~\ref{th2:thrmKreiss} (Kreiss).
In particular, the low computational cost of the following procedure is very appealing for the study of parameterised IBVP's, see Section~\ref{sec2:numerical}.

\begin{procedure}[Uniform Kreiss-Lopatinskii Condition check]\label{proc2:numericalprocedure} 
There are two different cases:
\begin{itemize}[label = $\bullet$]
    \item if $0 \in \courbe$, then there exists $z_0 \in \S$ such that $\DKLindep(z_0) = 0$.
    \item if $0 \notin \courbe$, $\DKLindep$ does not vanish on $\S$ and it has $r - \ind_{\courbe}(0)$ zeros in $\U$ by Theorem~\ref{thrm2:nbrzerodet}. It follows that if $\ind_{\courbe}(0) = r$ then the scheme is stable. Otherwise the scheme is unstable.
\end{itemize}
\end{procedure}
In summary, by Theorem~\ref{th2:thrmKreiss} (Kreiss) and since Uniform Kreiss-Lopatinskii Condition is fulfilled if and only if the Kreiss-Lopatinskii determinant has no zero in $\overline{\U}$, Method~\ref{proc2:numericalprocedure} can be used to conclude that the scheme is stable or not. Some illustrations for the O3 scheme and the fifth-order Lax-Wendroff scheme follow in Section~\ref{sec2:numerical}.

\section{Proof of Theorem~\ref{thrm2:mainholocont} and Corollary~\ref{thrm2:nbrzerodet}} \label{sec2:proof}

\subsection{Constant-recursive sequence of order $r$}

For each $z\in \U$, we denote $\PP_z$ the polynomial linked to the characteristic equation~\eqref{eq2:eqcharac}, i.e.
\begin{equation}\label{eq2:poleqcharacP}\PP_z(\kappa) = a_p \kappa^{r+p} + \cdots + a_1 \kappa^{r+1} + (a_0 - z)\kappa^r + a_{-1}\kappa^{r-1} + \cdots + a_{-r+1}\kappa + a_{-r}.\end{equation}
By Lemma~\ref{lem2:hersh} (Hersh), the polynomial $\PP_z(\kappa)$ can be factorized into two polynomials: one with the~$r$ roots in $\D$, denoted $\RR_z(\kappa)$ and one with the $p$ roots in $\U$, denoted $\QQ_z(\kappa)$. We know that the coefficients of $\PP_z$ are holomorphic in $z$. We already said that the basis~\eqref{eq2:basis} is not holomorphic because the roots $\kappa$ are not. In the next result, we prove that the symmetric functions of the $r$ roots $\kappa$ living in $\D$ are indeed holomorphic in $\U$, in other words, the coefficients of $\RR_z$ are holomorphic in~$\U$.

\begin{lemma}\label{lem2:coeffholo}
    For all $z \in \U$, the polynomial $\RR_z(X) =\prod_{j=1}^r (X-\kappa_j(z))$ has holomorphic coefficients in $\U$, where $(\kappa_j(z))_{j=1}^r$ are the $r$ roots (with multiplicity) in $\D$ of~\eqref{eq2:eqcharac}.
\end{lemma}

\begin{proof}
    We use the Dunford-Taylor formula with $\CC_z$ the companion matrix of the polynomial~\eqref{eq2:eqcharac}:
    \begin{equation*}
        \Pi(z) = \dfrac{1}{2\pi} \int_{\S} (\zeta I_{r+p} - C_z)^{-1} d\zeta
    \end{equation*}
    It is the projection along  $\EEs(z) = \ker \prod_{j = 1}^r  (C_z - \kappa_j(z))$ onto $\EEu(z) = \ker \prod_{j = r+1}^{r+p}  (C_z - \kappa_j(z))$ where $(\kappa_j(z))_{j=r+1}^{r+p}$ are the roots of \eqref{eq2:eqcharac} in $\U$, because $(\kappa_j(z))_{j=1}^{r}$ are surrounded by $\S$ and $(\kappa_j(z))_{j=r+1}^{r+p}$ are not. The projector $\Pi(z)$ is holomorphic on $\U$ since it is a holomorphic parameter integral. We have $C_z \circ \Pi(z)_{|\EEu(z)} = 0$ and $C_z \circ \Pi(z)_{|\EEs(z)} = C_z$, then the characteristic polynomial of $C_z \circ \Pi(z)$ is $X^p\RR_z(X)$ because $\C^{r+p} = \EEs(z) \oplus \EEu(z)$. The function $z \mapsto C_z \circ \Pi(z)$ is holomorphic on $\U$, then the coefficients of its characteristic polynomial are too. This concludes the proof.
\end{proof}

\subsection{Hermite interpolation}
To prove the holomorphic properties of $\DKLindep$, by \eqref{eq2:DKLavecKK}, it is sufficient to study the function $z \mapsto K_{0,m-1}(z)K_{0,r-1}^{-1}(z)$. To simplify this study, for $z \in \overline{\U}$, we introduce on the linear map 
\begin{equation}\label{eq2:defPhiZ}\HH_z : Q \in \C_{m-1}[X] \mapsto \HH_z(Q) \in \C_{r-1}[X]\end{equation}
where $\HH_z(Q)$ is the Hermite interpolation polynomial of degree less than $r-1$ defined by the value $(Q(\kappa_1(z)),\dots, Q^{(\mult_1-1)}(\kappa_1(z)),Q(\kappa_2(z)), \dots, Q^{(\mult_2-1)}(\kappa_2(z)), \dots, Q^{(\mult_\nbrmult-1)}(\kappa_\nbrmult(z)))$, where the $\kappa$'s are the same as in~\eqref{eq2:basis}.  The link between the matrix $K_{0,m-1}(z)K_{0,r-1}^{-1}(z)$ and $\varphi_z$ is given by:

\begin{lemma}\label{lem2:KKtranspo}
    For all $z \in \overline{\U}$, the transpose of the matrix $K_{0,m-1}(z)K_{0,r-1}^{-1}(z)$ is the representation in the canonical basis of $\HH_z$ defined in~\eqref{eq2:defPhiZ}.
\end{lemma}

\begin{proof}
The Hermite interpolation make appear the following matrix
\begin{center}
    \begin{tikzpicture}
        \node (A) at (-0.3,0) {
            $H_{0,j}(z) = \begin{pmatrix}
                1& 0& 0& \cdots                         &  1           & 0      & \cdots    & \cdots & 1        & 0 & \cdots  \\
    \kappa_1    &   1   & 0      & \cdots  &\kappa_2    & 1 & \cdots     & \cdots & \kappa_\nbrmult& 1 & \cdots  \\
    \kappa_1^2 & 2\kappa_1    & 2  & \cdots &  \kappa_2^2 & 2\kappa_2   & \cdots & \cdots & \kappa_\nbrmult^2   & 2 \kappa_\nbrmult & \cdots \\
    \kappa_1^3 & 3 \kappa_1^2   & 6 \kappa_1  & \cdots & \kappa_2^3 & 3 \kappa_2^2   & \cdots  & \cdots & \kappa_\nbrmult^3   & 3 \kappa_\nbrmult^2 &  \cdots \\
    \kappa_1^4 & 4 \kappa_1^3   & 12 \kappa_1^2 & \cdots  & \kappa_2^4 & 4 \kappa_2^3  & \cdots  & \cdots & \kappa_\nbrmult^4   & 4 \kappa_\nbrmult^3 & \cdots \\
    \vdots      & \vdots        & \vdots &      \cdots       & \vdots        & \vdots &      \cdots & \cdots &  \vdots &      \vdots & \cdots & ~ \\
    \kappa_1^{j} & j\kappa_1^{j-1} &j(j-1) \kappa_1^{j-2} & \cdots & \kappa_2^{j} & j\kappa_2^{j-1}  & \cdots & \cdots & \kappa_\nbrmult^{j}  &  j\kappa_\nbrmult^{j-1} & \cdots 
            \end{pmatrix}$.
        };
        \node (B) at (-2.85,-2.1) {$\underbrace{\phantom{00000000000000000000000000}}_{\mult_1 \text{ columns linked to }\kappa_1}$};
        \node (C) at (1.2,-2.1) {$\underbrace{\phantom{00000000000000}}_{\mult_2 \text{ columns linked to }\kappa_2}$};
        \node (D) at (4.9,-2.1) {$\underbrace{\phantom{00000000000000}}_{\mult_\nbrmult \text{ columns linked to }\kappa_\nbrmult}$};
    \end{tikzpicture}
\end{center}
The representation of $z\mapsto \HH_z$ in the canonical basis is $(H_{0,m-1}(z)H_{0,r-1}^{-1}(z))^T$. Besides, there exists an invertible matrix $M(z) \in \mathcal M_r(\C)$ such that $K_{0,j}(z) = H_{0,j}(z)M(z)$. Therefore, we have $K_{0,m-1}(z)K_{0,r-1}^{-1}(z) = H_{0,m-1}(z)M(z)M(z)^{-1}H_{0,r-1}^{-1}(z)=H_{0,m-1}(z)H_{0,r-1}^{-1}(z)$. The result follows.
\end{proof}

\begin{prop}\label{pr2:holomDKL}
    The function $z \mapsto \HH_z$ is holomorphic on $\U$.
\end{prop}

\begin{proof}
    For all $k \in \interval{0}{m-1}$, we want every coefficient of the polynomial $\HH_z(X^k)$ to be holomorphic on $\U$. Writing $\HH_z(X^k)(x) = \sum_{j=0}^{r-1} \alpha_{j,k}(z) x^j$, we know that \begin{equation}\label{eq2:coeffderiv}\forall j \in \interval{0}{r-1},\quad j! \alpha_{j,k}(z) = \partial_x^j \HH_z(X^k)(x)_{|x=0}.\end{equation}
    By the error of Hermite interpolation (see \cite{Hermite1877}), we have
    \begin{equation}\label{eq2:Hermiteerror}
        \HH_z(X^k)(x) - x^k = \dfrac{1}{2i\pi} \int_{\S} \dfrac{\zeta^k \RR_z(x)}{(x - \zeta)\RR_z(\zeta)}d\zeta\end{equation}
    where $\RR_z(X)$ is defined in Lemma~\ref{lem2:coeffholo}.
    Differentiating equation~\eqref{eq2:Hermiteerror} (with the Leibniz product rule), one obtains
    \begin{equation}\label{eq2:coeffKK}
        j! \alpha_{j,k}(z)  = k! \delta_{k}^j + \sum_{s=0}^j \binom{j}{s} \RR_z^{(j-s)}(0) \dfrac{1}{2i\pi} \int_{\S} \dfrac{- s! ~\zeta^{k-s-1}}{\RR_z(\zeta)}d\zeta.
    \end{equation}
    By Lemma~\ref{lem2:coeffholo} and the holomorphicity of parameter-dependent integrals, the function $z \mapsto \alpha_{j,k}(z)$ is holomorphic on $\U$ for all $j\in \interval{0}{r-1}$ and $k\in \interval{0}{m-1}$. The proof is now complete.
\end{proof}

\begin{prop}\label{pr2:contDKL}
    The function $z \mapsto \HH_z$ is continuous on $\overline{\U}$.
\end{prop}

\begin{proof}
    Because Lemma~\ref{lem2:hersh} (Hersh) does not hold anymore for $z\in \S$, the roots $\kappa(z)$ of characteristic equation~\eqref{eq2:eqcharac} can be on the unit circle $\S$. To prove the continuity of $z \mapsto \alpha_{j,k}(z)$, we use equation~\eqref{eq2:coeffKK} but replacing $\S$ by $\Se \egdef  \{z\in \C, |z|=1+\eps\}$ for $\eps >0$. Using the continuity of parameter-dependent integrals and the continuity of the roots $\kappa(z)$ of characteristic equation~\eqref{eq2:eqcharac}, we obtain the continuity of the coefficients of $\RR_z$ and thus the function $z \mapsto \alpha_{j,k}(z)$ is continuous on $\overline{\U}$ for all $j\in \interval{0}{r-1}$ and $k\in \interval{0}{m-1}$. The proof is now complete.
\end{proof}

\begin{prop}\label{pr2:borneDKL}
    The function $z \mapsto \HH_z$ is bounded on $\overline{\U}$.
\end{prop}

\begin{proof}
    Equation~\eqref{eq2:coeffKK} can give a bound of every components of $\H_{0,m-1}(z) \H_{0,r-1}^{-1}(z)$. Indeed, using Rouché's theorem, as in the proof of Lemma~\ref{lem2:hersh} (Hersh), we can see that for $|z|>R$ for a certain~$R$, all the roots $\kappa(z)$ of the characteristic equation~\eqref{eq2:eqcharac} satisfy $|\kappa(z)|<\frac 1 2$. Then, for $|z|>R$, one can have
    $$\left |\dfrac{1}{2i\pi} \int_{\S} \dfrac{- s! ~\zeta^{k-s-1}}{\RR_z(\zeta)}d\zeta \right | = \left |\dfrac{s!}{2i\pi} \int_{0}^{2\pi} \dfrac{e^{i\theta (k-s-1)}}{\prod_{j=1}^r (e^{i\theta} - \kappa_j(z))}d\theta \right | \< \dfrac{s!}{2\pi} \int_{0}^{2\pi} \dfrac{1}{|1 - \frac 1 2 |^r} d\theta \< s! 2^r.$$
    By Gauss-Lucas theorem, the roots of all the derivatives of $\RR_z$ are in $\D$ for all $|z|>1$, it follows that $\RR_z^{(j-s)}(0)$ is bounded independently of $z$. Then for $|z|>R$, the quantity $\H_{0,m-1}(z) \H_{0,r-1}^{-1}(z)$ is bounded. Moreover, by Proposition~\ref{pr2:contDKL},  the quantity $\H_{0,m-1}(z) \H_{0,r-1}^{-1}(z)$ is bounded on the compact set $\{1\<|z|\<R\}$. The proof is now complete.
\end{proof}

\subsection{Conclusion}

\begin{proof}[Proof of Theorem~\ref{thrm2:mainholocont}]
    By Lemma~\ref{lem2:KKtranspo}, the continuity and holomorphicity properties of $z \mapsto \varphi_z$ provided in Propositions \ref{pr2:holomDKL} and \ref{pr2:contDKL} are shared by the function $z\mapsto K_{0,m-1}(z)K_{0,r-1}^{-1}(z)$. The expression of the intrinsic determinant~\eqref{eq2:DKLavecKK} concludes the proof.
\end{proof}
%
%
%
The next proof is close to the proof of the Corollary 15 of \cite{Boutin22}.
We reproduce it here for completeness.

\begin{proof}[Proof of Theorem~\ref{thrm2:nbrzerodet}]
    Let us define the following function
    \[\fun{\DKLindepinvers}{z}{\D^*}{\DKLindep(1/z)}{\C.}\]
    By Theorem~\ref{thrm2:mainholocont}, the function $\DKLindepinvers$ is meromorphic on $\D$ with one only pole in $0$ and is continuous on~$\overline{\D}\setminus \{0\}$.
    By Proposition~\ref{pr2:borneDKL}, the function $z\mapsto \H_{0,m-1}(1/z) \H_{0,r-1}(1/z)^{-1}$ is bounded on $\overline{\D}$, it follows that $0$ is a pole of order $r$ of the function $\DKLindepinvers$.
    The residue theorem applied on $\DKLindepinvers$ with the path $\S$ gives the following equality:
    \[\Ind_{\DKLindepinvers(\S)}(0) = \#\mathrm{zeros}_{\DKLindepinvers}(\D) - \#\mathrm{poles}_{\DKLindepinvers}(\D).\]
    It follows that 
    \[\#\mathrm{zeros}_{\DKLindep}(\U) = r - \Ind_{\DKLindep(\S)}(0).\]
\end{proof}


\section{Numerical results}\label{sec2:numerical}

\subsection{New formulation of $\DKLindep$}

In~\cite{Boutin22}, an explicit formula of the Kreiss-Lopatinskii determinant is given. Unfortunately to reduce the boundary matrix $\Bbord$, the characteristic equation of degree $r+p$ is used to find a final matrix of size $r \times (r+p)$ and, since in the present analysis $p\neq0$, the matrix is not square. To skirt that problem  we will use the polynomial $\RR_z$ defined in Lemma~\ref{lem2:coeffholo} instead of using the complete characteristic equation~\eqref{eq2:poleqcharacP}. It reads also
\[\RR_z(X) = \prod_{j=1}^{r} (X-\kappa_j(z)) = X^r + \coefpol_{r-1}(z) X^{r-1} + \cdots + \coefpol_1(z) X + \coefpol_0(z)\]
where $(\coefpol_j(z))_j$ are the symmetric functions of $(\kappa_j(z))_j$.

Because $\widetilde{U}_j(z)$ is in $\Es$ and can be expressed in the basis~\eqref{eq2:basis}, we have, for all $j\in \N$,
\begin{equation}\label{eq2:eqcharacRz}\widetilde{U}_{j+r}(z) + \coefpol_{r-1}(z) \widetilde{U}_{j+r-1}(z) + \cdots + \coefpol_1(z) \widetilde{U}_{j+1}(z) + \coefpol_0(z)\widetilde{U}_{j}(z) = 0.\end{equation}

\noindent\textbf{Notation.} We note, for all $j \in \N$, $\Uvec{j}$ the vector $(\widetilde{U}_0(z) \cdots \widetilde{U}_{j}(z) )^T$ of size $j+1$.

\begin{prop}\label{pr2:constructionB}Let $\Bbord \in \mathcal M_{r,m}(\C)$. There exists a function $\widetilde{\Bbord} : \C^r \to \mathcal M_{r,r}(\C)$  
constructible such that, for all $z \in \overline{\U}$, we have
\begin{equation}\label{eq2:transfoalgo}\Bbord \Uvec{m-1}= \widetilde{\Bbord}(\sigma_0(z), \dots, \sigma_{r-1}(z)) \Uvec{r-1}
\end{equation}
where $(\widetilde{U}_j(z))_j$ satisfies \eqref{eq2:eqcharacRz} for all $j\in \N$.
\end{prop}
\noindent By "constructible function", we mean here that we establish a computable algorithm to get the matrix $\widetilde{\Bbord}(\sigma_0(z), \dots, \sigma_{r-1}(z))$. This algorithm, based on a Gaussian elimination, is fully described in the following proof.

\begin{proof} For $z\in \overline{\U}$ and $\varsigma_0 = \sigma_0(z), \dots, \varsigma_{r-1} = \sigma_{r-1}(z)$.
    By a descending induction on $j$ between $m-1$ to $r-1$, we construct a matrix $\Bbord_j(\varsigma_0, \dots, \varsigma_{r-1})\in \mathcal M_{r,j+1}(\C)$ such that  
    \begin{equation*}\Bbord \Uvec{m-1}= \Bbord_j(\varsigma_0, \dots, \varsigma_{r-1}) \Uvec{j}.
    \end{equation*}
    \emph{Initialization:} if $j = m-1$ then one can take $\Bbord$ for the matrix $\Bbord_{m-1}(\varsigma_0, \dots, \varsigma_{r-1})$.\\ 
    \emph{Induction:} we assume the induction hypotheses for some $j \in \interval{m-1}{r}$ and we want to prove the result for $j-1$. By equation~\eqref{eq2:eqcharacRz}, we have $\Uvec{j} = \Pbord_j \Uvec{j-1}$ where 
    \[\Pbord_j =  \begin{pmatrix}1 & & & \\ &\ddots & & \\ & & \ddots& \\ & & & 1\\ (0)& -\varsigma_0& \cdots & -\varsigma_{r-1}\end{pmatrix} \in \mathcal M_{j+1,j}(\C).\] We define $\Bbord_{j-1}(\varsigma_0, \dots, \varsigma_{r-1})= \Bbord_j(\varsigma_0, \dots, \varsigma_{r-1}) \Pbord_j \in \mathcal M_{r,j}(\C)$ then we have 

    \[\Bbord_{j-1} \Uvec{j-1} = \Bbord_j \Pbord_j\Uvec{j-1} = \Bbord_j\Uvec{j} = \Bbord \Uvec{m-1}.\]
    \noindent\emph{Conclusion:} we define $\widetilde{\Bbord}$ by $\Bbord_{r-1}$.\\
    The function $\widetilde{\Bbord}$ is easily computable because $(\Pbord_j)_j$ are just matrices of Gaussian elimination.
\end{proof}

By~\eqref{eq2:transfoalgo}, the intrinsic Kreiss-Lopatinskii determinant can be written 
\begin{equation}\label{eq2:newformuladetKL}\DKLindep(z) = \dfrac{\det(zK_{0,r-1}(z) - \Bbord K_{0,m-1}(z))}{\det K_{0,r-1}(z)} = \dfrac{\det(zK_{0,r-1}(z) - \widetilde{\Bbord} K_{0,r-1}(z))}{\det K_{0,r-1}(z)}  = \det(zI_r - \widetilde{\Bbord}).
\end{equation}
The matrix $\widetilde{\Bbord}$ from Proposition~\ref{pr2:constructionB} depends on coefficients $(\sigma_j(z))_j$. By~\eqref{eq2:transfoalgo}, we have \[\widetilde{\Bbord}(\coefpol_0(z), \dots, \coefpol_{r-1}(z)) = \Bbord K_{0,m-1}(z)K_{0,r-1}^{-1}(z).\]
Using any computer algebra system, we can compute the matrix $\widetilde{\Bbord}(\varsigma_0, \dots, \varsigma_{r-1})$ from the matrix~$\Bbord$ and then compute the coefficients $(\coefpol_{j}(z))_j$ and replace $\varsigma_j$ by $\sigma_j(z)$ for all $j \in \interval{0}{r-1}$. It provides that the computation of $\widetilde{\Bbord}(\varsigma_0, \dots, \varsigma_{r-1})$ can be done only once and then apply for different~$z$ (and so different  $(\coefpol_{j}(z))_j$).

With \eqref{eq2:newformuladetKL}, we again find the holomorphic property of $\DKLindep$ by Lemma~\ref{lem2:coeffholo} which states that $z \mapsto \coefpol_j(z)$ is holomorphic on $\U$ for all $j \in \interval{0}{r-1}$. 

\subsection{Computation of $\DKLindep(\S)$}\label{sec2:numDeltaselection}

Let us fix a $z_0 \in \S$. To compute $(\coefpol_{j}(z_0))_j$, we need the $r$ roots $(\kappa_j(z_0))_j$ that come from the inside of the unit disk, see Remark~\ref{rem2:rootsfrominsideoutside}. 
By the continuity of the roots of polynomial $P_{z_0}$ defined in~\eqref{eq2:poleqcharacP} with respect to the parameter $z_0$, for each $\kappa_0(z_0)$ of multiplicity $\mult$ on the unit circle, for a sufficiently small $\varepsilon>0$, there exists $\eta>0$ such that for all $z \in B(z_0,\eta)$, the polynomial $P_z$ has exactly $\mult$ roots with multiplicity in  $B(\kappa_0(z_0), \varepsilon)$. The explicit value of $\eta$ is given in the following statement.

\begin{lemma}\label{lem2:rootcounting}
    Let $z_0$ be on the unit circle. Let $\kappa_0(z_0)\in \S$ be a root of multiplicity $\mult$ of the polynomial~$P_{z_0}$ defined  in~\eqref{eq2:poleqcharacP}.
    Let $\varepsilon>0$ be such that $\kappa_0(z_0)$ is the only root of $P_{z_0}$ in $\overline{B(\kappa_0(z_0), \varepsilon)}$ and set \[\eta = (1+\varepsilon)^{-r}\underset{\kappa \in \partial B(\kappa_0(z_0), \varepsilon)}{\min} |P_{z_0}(\kappa)|.\]
    Then for all $z \in B(z_0, \eta)$, the polynomial $P_z$ has exactly $\mult$ roots with multiplicity in  $B(\kappa_0(z_0), \varepsilon)$.
\end{lemma}

The proof of Lemma~\ref{lem2:rootcounting} is a consequence of Rouché's theorem, comparing the number of zeros between $P_{z_0}$ and $P_z$ for $z$ close to $z_0$ in $B(\kappa_0(z_0), \eps)$, the details of the proof are not given here, but we refer to \cite{Lang13} for a proof of Rouché's theorem.
This lemma is illustrated in Figure~\ref{fig2:rootcounting} where $\kappa_0(z_0)$ is of multiplicity $3$. The black points are related to $z_0$ and $(\kappa_j(z_0))_j$, and the gray points are related to $z$ and $(\kappa_j(z))_j$.

\begin{figure}
    \centering
    \begin{tabular}{cc}
        Where $z$ lives & Where $\kappa$ lives \\
\begin{tikzpicture}
    \draw[white, fill] (-2.2,-2.2) rectangle (2.2,2.2);
    \draw(0,0) circle (2);

    \draw  (1.414,1.414) node{$+$};
    \draw (1.25,1.25) node{$z_0$};
    \draw[->,>=stealth, gray] (1.414,1.414) arc (310:320:2) ;
    \draw[gray]  (1.414+0.35/1.414,1.414+0.35/1.414) node{$+$};
    \draw[gray] (1.8,1.5) node{$z$};
    \draw[dashed] (1.414,1.414) circle (0.7);
    \draw  (0,2) node[below]{$\S$};
    \draw[<->] (1.414,1.414) -- (1.414,1.414+0.7);
    \draw (1.3,1.414+0.3) node{$\eta$};
\end{tikzpicture}

& \begin{tikzpicture}
    \draw[white, fill] (-2.2,-2.2) rectangle (4.2,2.2);
    \draw (0,0) circle (2);
    \draw  (0.4077563762011518,0.7969276234786205) node{$\times$} node[above right]{$\kappa(z_0)$};
    \draw[->,>=stealth, gray] (0.4077563762011518,0.7969276234786205) arc (10:30:2) ;
    \draw[gray]  (0.1,1.6) node{$\times$};
    \draw[->,>=stealth, gray] (0.4077563762011518,0.7969276234786205) arc (90:110:2) ;
    \draw[gray]  (-0.4,0.6) node{$\times$};
    \draw  (-1.5,-0.3) node{$\times$} node[above right]{$\kappa(z_0)$};
    \draw[->,>=stealth, gray] (-1.5,-0.3) arc (220:240:2) ;
    \draw[gray]  (-0.8,-0.8) node{$\times$};
    \draw  (1.414,-1.414) node{$\times$};
    \draw[->,>=stealth, gray] (1.414,-1.414) arc (15:32:2) ;
    \draw[->,>=stealth, gray] (1.414,-1.414) arc (360:350:4) ;
    \draw[->,>=stealth, gray] (1.414,-1.414) arc (290:310:2) ;
    \draw (1.414,-1.55) node{$\kappa_0(z_0)$};
    \draw[dashed] (1.414,-1.414) circle (0.9);
    \draw[<->] (1.414,-1.414) -- (1.414-0.9,-1.414);
    \draw (1.414-0.45,-1.3) node{$\eps$};
    \draw[gray] (1.3,-2.2) node{$\times$};
    \draw[gray] (1.1,-0.8) node{$\times$};
    \draw[gray] (2,-1) node{$\times$};
    \draw[->,>=stealth, gray] (1.9,1.5) arc (90:70:2) ;
    \draw[gray]  (2.7,1.3) node{$\times$} node[below]{$\kappa(z)$};
    \draw  (1.9,1.5) node{$\times$} node[above]{$\kappa(z_0)$};
\end{tikzpicture}
\end{tabular}
\caption{Illustration of Lemma~\ref{lem2:rootcounting}}\label{fig2:rootcounting}
\end{figure}
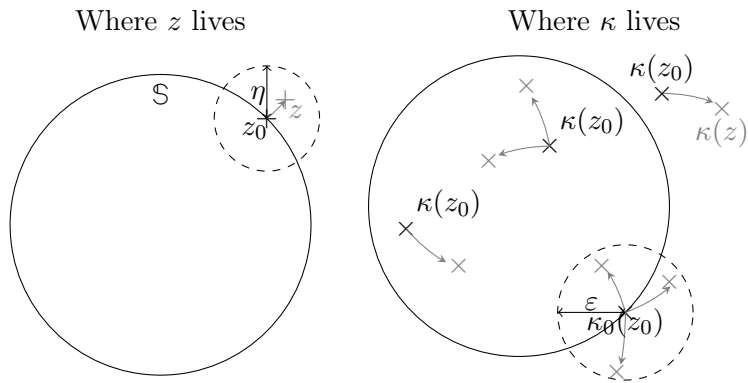

From the numerical point of view, for a multiple root $\kappa_0(z_0)$, one can take the smallest distance between two roots of $P_{z_0}$ as $\eps$, take $z = (1+ \frac \eta 2)z_0 \in \U$. The value of $\eta$  is obtained discretizing the circle of radius $\eps$ centered in $\kappa_0(z_0)$.
By Lemma~\ref{lem2:hersh} (Hersh), there is no roots of $P_z$ on the unit circle, then one can count the roots in $B(\kappa_0(z_0), \eps)\cap \D$ and $B(\kappa_0(z_0),\eps)\cap \U$ to know the number of roots linked to $\kappa_0(z)$ that come from the inside and the outside of the unit disk. 
After selecting the roots $(\kappa_j(z))_j$ that come from the inside of the unit disk, one may compute their symmetric functions $(\coefpol_{j}(z))_j$. By replacing the formal variables $(\varsigma_j)_j$ of $\widetilde{\Bbord}(\varsigma_0,\dots, \varsigma_{r-1})$ with $(\coefpol_{j}(z))_j$, one may compute $\DKLindep(z)$ with expression~\eqref{eq2:newformuladetKL}.
Instead of computing $(\kappa_j(z))_j$ for each $z$ on the unit circle independently, one may use the continuity of $(\kappa_j(z))_j$ with respect to $z$ in order to describe the movement of the roots $(\kappa_j(z))_j$ for $z\in \S$.
After drawing the Kreiss-Lopatinskii curve, the winding number has to be computed in order to use Method~\ref{proc2:numericalprocedure}. To do so, we use the geometric algorithm proposed by Garc\'ia Zapata and D\'iaz Mart\'in in \cite{Zapata12} and \cite{Zapata14}.

\subsection{Boundary condition: reconstruction procedure}\label{sec2:boundarycondition}
To define the boundary condition, we use the reconstruction procedure explained in \cite{Dakin18}. The framework is the advection equation with a misalignment between the space boundary and the discrete grid points
\begin{align}\label{eq2:advectionsigma}
    \begin{cases}
     \partial_t u + a \partial_x u = 0, &  t\>0, x \in [\borneinfintervalspatial,1],\\
    u(t,\borneinfintervalspatial) = g(t), & t\>0, \\
     u(0,x) = f(x), & x\in [\borneinfintervalspatial,1]. 
    \end{cases}
\end{align}
Without loss of generality, we can assume that $\borneinfintervalspatial = \sigma \dx$ with $\sigma \in [-\frac 1 2, \frac 1 2[$ (as it is explained in~\cite{Boutin22}). 
Let us introduce $x_j$ for $j\dx$ and $t^n$ for $n\dt$ when $j\>-r$ and $n\>0$. Let $n\in \N$ be a fixed time. The solution $u$ of~\eqref{eq2:advectionsigma} (assumed here to be smooth enough) satisfies
\begin{equation}\label{eq2:formintegsolu}
    \frac{1}{\dx} \int_{x_j-\frac{\dx}{2}}^{x_j+\frac{\dx}{2}} u(t^n,y)dy = \frac{1}{\dx} \int_{x_j-\frac{\dx}{2}}^{x_j+\frac{\dx}{2}} \sum_{k = 0}^{d-1} \partial_x^k u(t^n,\borneinfintervalspatial) \frac{(y-\borneinfintervalspatial)^k}{k!}dy + O(\dx^d)
\end{equation}
using a Taylor expansion of order $d$. 
Now, let us take a solution $(U_j^n)_{j\>0}$ of a scheme of the form~\eqref{eq2:stdscheme} approximating $u$.
Using \eqref{eq2:formintegsolu}, we want to define the $r$ ghost points $(U_j^n)_{-r\<j\<-1}$. The approximation of equation~\eqref{eq2:formintegsolu} reads, for $j\>-r$, 
\begin{equation}\label{eq2:reconstructionUjn}U_j^n \approx \sum_{k = 0}^{d-1} \partial_x^k u(t^n,\borneinfintervalspatial) \left (\dfrac{(j + \tfrac{1}{2}-\sigma)^{k+1}}{(k+1)!} - \dfrac{(j - \tfrac{1}{2}-\sigma)^{k+1}}{(k+1)!}\right ).
\end{equation}

On the one hand, we use the PDE to convert space derivatives into time derivatives until an index $k_d<d$. The index $k_d$ allows us to know only the first derivatives of the boundary datum $g$ and use extrapolation for the rest. For the advection equation, we have, for all $k\<k_d$,
\[\partial_x^k u(t^n,\borneinfintervalspatial) = (-a)^{-k}\partial_t^k u(t^n,\borneinfintervalspatial) = (-a)^{-k} g^{(k)}(t^n).\]
Equation~\eqref{eq2:reconstructionUjn} becomes, for $j\>-r$,
\begin{multline}\label{eq2:reconstructionUjn2}U_j^n \approx \sum_{k=0}^{k_d} (-a)^{-k}g^{(k)}(t^n) \left (\dfrac{(j + \tfrac{1}{2}-\sigma)^{k+1}}{(k+1)!} - \dfrac{(j - \tfrac{1}{2}-\sigma)^{k+1}}{(k+1)!}\right ) \\+\sum_{k = k_d+1}^{d-1} \partial_x^k u(t^n,\borneinfintervalspatial) \left (\dfrac{(j + \tfrac{1}{2}-\sigma)^{k+1}}{(k+1)!} - \dfrac{(j - \tfrac{1}{2}-\sigma)^{k+1}}{(k+1)!}\right ).
\end{multline}

On the other hand, we need to define $(\partial_x^ku(t^n,\borneinfintervalspatial))_{k=k_d+1}^{d-1}$, but using \eqref{eq2:reconstructionUjn2} for $j \in \interval{0}{d-k_d-2}$, we can deduce the unknowns $(U_j^n)_{-r\<j\<-1}$. Writing $\Uvecmoins = (U_{-r}^n,\dots, U_{-1}^n)^T$, $\Uvecplus = (U_0^n,\dots, U_{d-k_d-2}^n)^T$ and $\Theta^n = (\partial_x^{k_d+1} u(t^n,\borneinfintervalspatial), \dots, \partial_x^{d-1} u(t^n,\borneinfintervalspatial))^T$, we have a condensed formulation of~\eqref{eq2:reconstructionUjn2}:
\begin{equation}\label{eq2:systemTheta}\begin{cases}\begin{array}{rcl}\Uvecmoins & =&  \mathcal S_{-}^n + \mathcal Y_{-}\Theta^n, \\
    \Uvecplus &=&  \mathcal S_{+}^n + \mathcal Y_{+}\Theta^n,
\end{array}\end{cases}\end{equation}
where $\mathcal S_-^n \in \R^{r}$, $\mathcal S_+^n \in \R^{d-k_d-1}$, $\mathcal Y_- \in \mathcal M_{r,d-k_d-1}(\R)$ and $\mathcal Y_+ \in \mathcal M_{d-k_d-1}(\R)$ with
\[\begin{cases}
    (\mathcal S_-^n)_i =  \sum_{k=0}^{k_d} (-a)^{-k}g^{(k)}(t^n) \left (\dfrac{(-i + \tfrac{1}{2}-\sigma)^{k+1}}{(k+1)!} - \dfrac{(-i - \tfrac{1}{2}-\sigma)^{k+1}}{(k+1)!}\right ) &  \hspace{-9pt}  \text{for }i \in \interval{1}{r}\\
    (\mathcal S_+^n)_i =  \sum_{k=0}^{k_d} (-a)^{-k}g^{(k)}(t^n) \left (\dfrac{(i-1 + \tfrac{1}{2}-\sigma)^{k+1}}{(k+1)!} - \dfrac{(i-1 - \tfrac{1}{2}-\sigma)^{k+1}}{(k+1)!}\right ) & \hspace{-9pt}\text{for }i \in \interval{1}{d-k_d-1}\\
    (\mathcal Y_{-})_{i,j} = \left (\dfrac{(-(r-i+1) + \tfrac{1}{2}-\sigma)^{j+k_d+1}}{(j + k_d+1)!} - \dfrac{(-(r-i+1) - \tfrac{1}{2}-\sigma)^{j+k_d+1}}{(j+k_d+1)!}\right )& \hspace{-4mm}\begin{array}{l} \text{for }i \in \interval{1}{r},\\ j \in \interval{1}{d-k_d-1} \end{array} \\
    (\mathcal Y_{+})_{i,j} = \left (\dfrac{(i-1 + \tfrac{1}{2}-\sigma)^{j+k_d+1}}{(j+k_d+1)!} - \dfrac{(i-1 - \tfrac{1}{2}-\sigma)^{j+k_d+1}}{(j+k_d+1)!}\right ) &  \hspace{-1cm}  \text{for }i,j \in \interval{1}{d-k_d-1}
\end{cases}\]
%
%
Eliminating the space derivatives $\Theta^n$ in \eqref{eq2:systemTheta} gives us the following boundary condition:
\begin{equation}\label{eq2:boundaryreconstruction}\Uvecmoins  = \mathcal Y_- \mathcal Y_+^{-1}\Uvecplus +\mathcal S_-^n - \mathcal Y_- \mathcal Y_+^{-1}\mathcal S_+^n.\end{equation}
Equation~\eqref{eq2:boundaryreconstruction} is exactly the boundary equations \eqref{eq2:eqbord} of the scheme which define the $r$ ghost points of the scheme.
To write the boundary condition as equation \eqref{eq2:equationBordQuasiToep} with expression~\eqref{eq2:bordB}, we identify
\[
    B \overset{def}{=} \mathcal Y_- \mathcal Y_+^{-1} \text{ and }\begin{pmatrix}g_{-r}^n \\ \vdots \\ g_{-1}^n\end{pmatrix}\overset{def}{=} \mathcal S_-^n - \mathcal Y_- \mathcal Y_+^{-1}\mathcal S_+^n.
\]
As in \cite{Dakin18}, $\mathcal R^{d,k_d}$ denotes the reconstruction procedure where $d$ is the order of consistency of the method and $k_d$ the index when we change from time derivatives to extrapolation.
For example, the reconstruction procedure $\mathcal R^{3,0}$ for $r=2$ and $\sigma = 0.4$ leads to
\[\mathcal Y_{-} = \begin{pmatrix} -\frac{12}{5} & \frac{1753}{600} \\ -\frac{7}{5} & \frac{613}{600} \end{pmatrix}, \mathcal Y_{+} = \begin{pmatrix}-\frac{2}{5} & \frac{73}{600} \\ \frac{3}{5} & \frac{133}{600}  \end{pmatrix} \text{ and } B = \begin{pmatrix} \frac{1371}{97} & \frac{526}{97}  \\ \frac{554}{97} & \frac{143}{97} \end{pmatrix} \]

\subsection{Example of O3 scheme}

As it is done in \cite{Dakin18}, we want to find the stability area for the O3 scheme defined, for $j \in \N$ and $n\in \N$, by
\begin{equation}U_{j}^{n+1}  =  \left (\frac{\lambda^3}{6}  -  \frac{\lambda}{6}\right )U_{j-2}^n + \left (\lambda + \frac{\lambda^2}{2} - \frac{\lambda^3}{2}\right ) U_{j-1}^n + \left (1 - \frac{\lambda}{2} - \lambda^2 +\frac{\lambda^3}{2}\right ) U_j^n+\left (\frac{\lambda^2}{2} - \frac{\lambda^3}{6} - \frac{\lambda}{3}\right )U_{j+1}^n\end{equation}
The O3 scheme is a scheme with $r = 2$ and $p=1$ and is Cauchy-stable for $\lambda \in ]0,1]$.
%
%
The reconstruction $\mathcal R^{3,0}$ for the O3 scheme and $\sigma = 0.4$ leads to
\[B = \begin{pmatrix} \frac{1371}{97} & \frac{526}{97} & 0 \\ \frac{554}{97} & \frac{143}{97}&0 \end{pmatrix} \text{ and }\Bbord = \begin{pmatrix}\frac{180\lambda^2}{97} + \frac{277\lambda}{97} + 1& \frac{120\lambda^2}{97} + \frac{23\lambda}{97}&  0 \\ \frac{263\lambda^3}{582} + \frac{\lambda^2}{2} + \frac{14\lambda}{291}& \frac{217\lambda^3}{291} - \lambda^2 - \frac{217\lambda}{291} + 1& -\frac{\lambda^3}{6} + \frac{\lambda^2}{2} - \frac{\lambda}{3} \end{pmatrix}.\]
Using the reformulation~\eqref{eq2:newformuladetKL} of the Kreiss-Lopatinskii determinant, we have

\begin{align*}\widetilde{\Bbord}(\varsigma_0,\varsigma_1) & = \Bbord \begin{pmatrix}1 & 0 \\ 0  & 1 \\ - \varsigma_0 & -\varsigma_1\end{pmatrix}\\ & =\begin{pmatrix}\frac{180\lambda^2}{97} + \frac{277\lambda}{97} + 1& \frac{120\lambda^2}{97} + \frac{23\lambda}{97} \\ 
    \frac{(263+97\varsigma_0)\lambda^3}{582} + \frac{(1-\varsigma_0)\lambda^2}{2} + \frac{(14+97\varsigma_0)\lambda}{291}& 
    \frac{(434 + 97 \varsigma_1)\lambda^3}{582} - \frac{(2+\varsigma_1)\lambda^2}{2} - \frac{(217-97\varsigma_1)\lambda}{291} + 1\end{pmatrix}.
\end{align*}
%
%
For example, for $\sigma = 0.4$,  Figure~\ref{fig2:DKLO3R40} shows that the O3 scheme with $\mathcal R^{3,0}$ boundary is stable for $\lambda = 0.4$ (because $r - \Ind_{\DKLindep(\S)}(0) = 0$) and is unstable for $\lambda = 0.9$ (because $r - \Ind_{\DKLindep(\S)}(0) = 1$). 

\begin{figure}[!h]
    \includegraphics[trim = 2cm 1.4cm 2.5cm 2.4cm, clip, width = 10cm]{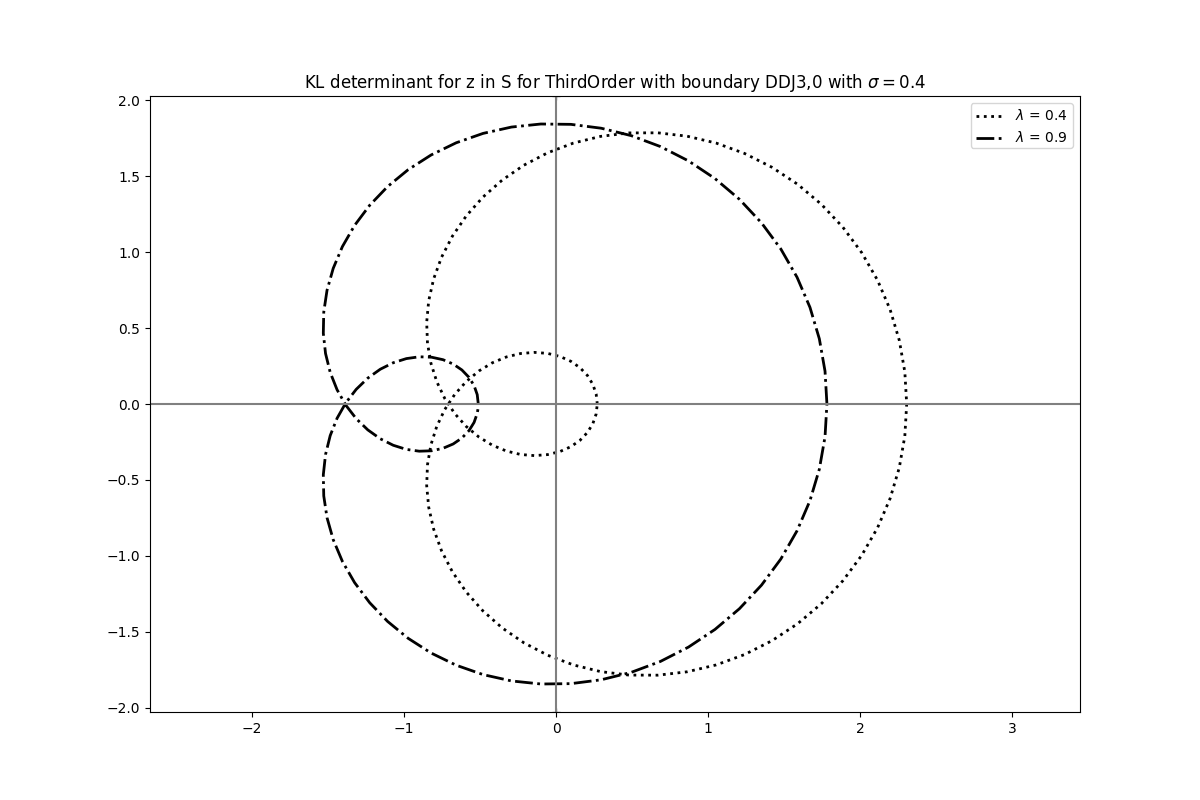}
    \caption{Curve $\DKLindep(\S)$ for O3 scheme for $\sigma = 0.4$, for $\lambda \in \{0.4,0.9\}$ with reconstruction boundary $\mathcal R^{3,0}$.\label{fig2:DKLO3R40}}
\end{figure}

\medskip

We can draw the same figure as the Figure 4 of \cite{Dakin18} but instead of using a computation of the spectral radius of the truncated quasi-Toeplitz matrix, we use our strategy of counting the number of instability modes, see Figure~\ref{fig2:asDDJ} which is much more reliable (since it is parameter free) and efficient.
In Figure~\ref{fig2:asDDJ}, every area stamped with 0 is a domain where the O3 scheme is stable. 
The odd pattern for very small $\lambda$ (approximatively between $0$ and $0.01$) of Figure~\ref{fig2:asDDJ} may be due to difficulties for computing the winding number. Indeed, for very small values of $\lambda$, the Kreiss-Lopatinskii determinant is really close to the origin and even with a refinement (for more details on this procedure see the next example), the computation of the winding number may become inaccurate, which is not a problem in practice since it would correspond to very small, then unusable, time steps.

\begin{figure}
\begin{center}
\begin{tabular}{cc}
    $\mathcal R^{3,0}$ & $\mathcal R^{3,1}$ \\
    \begin{tikzpicture}
        \node (A) at (0,0) {\includegraphics[trim = 1.5cm 0.7cm 2cm 1.45cm, clip, width = 8cm]{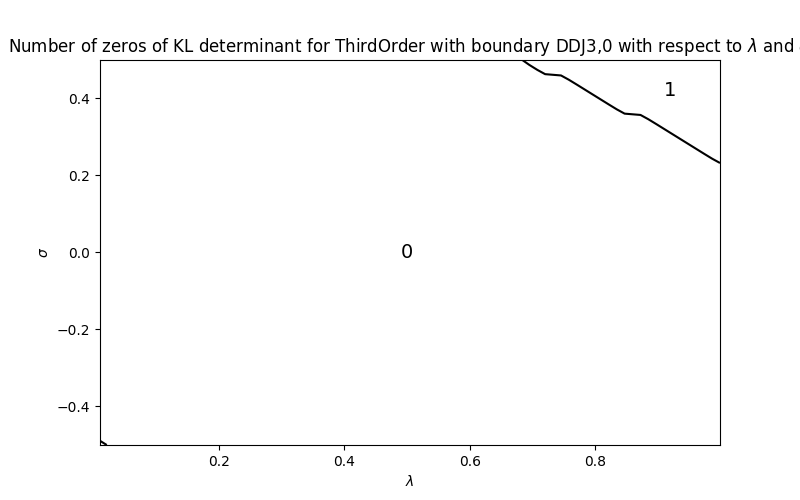}};
        
        \node (s) at (-3.75,2.4) {\scriptsize{$\sigma$}};
        \node (l) at (3.6,-2.4) {\scriptsize{$\lambda$}};
    \end{tikzpicture} 
    & \begin{tikzpicture}
        \node (A) at (0,0) {\includegraphics[trim = 1.5cm 0.7cm 2cm 1.45cm, clip, width = 8cm]{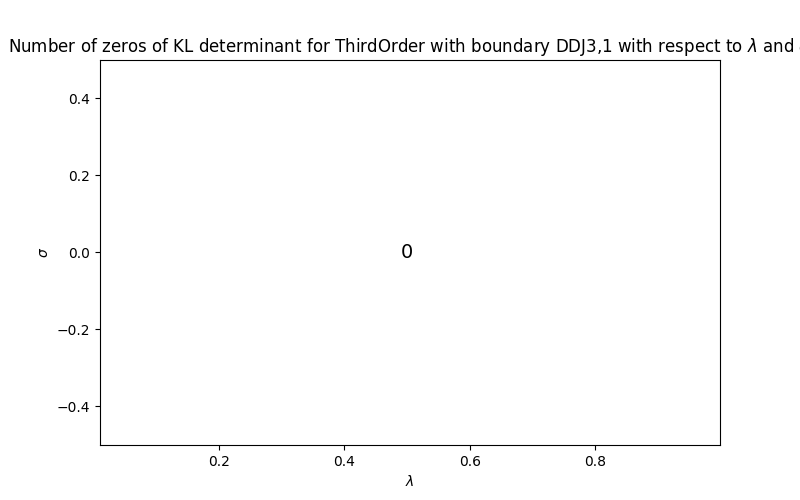}};
        
        \node (s) at (-3.75,2.4) {\scriptsize{$\sigma$}};
        \node (l) at (3.6,-2.4) {\scriptsize{$\lambda$}};
    \end{tikzpicture}\\ 
    $\mathcal R^{4,0}$ & $\mathcal R^{4,1}$ \\
    \begin{tikzpicture}
        \node (A) at (0,0) {\includegraphics[trim = 1.5cm 0.7cm 2cm 1.45cm, clip, width = 8cm]{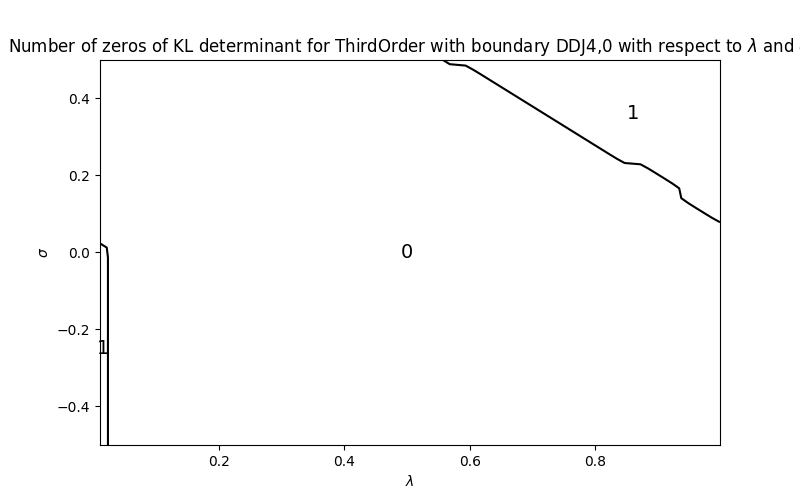}};
        
        \node (s) at (-3.75,2.4) {\scriptsize{$\sigma$}};
        \node (l) at (3.6,-2.4) {\scriptsize{$\lambda$}};
    \end{tikzpicture} 
    & \begin{tikzpicture}
        \node (A) at (0,0) {\includegraphics[trim = 1.5cm 0.7cm 2cm 1.45cm, clip, width = 8cm]{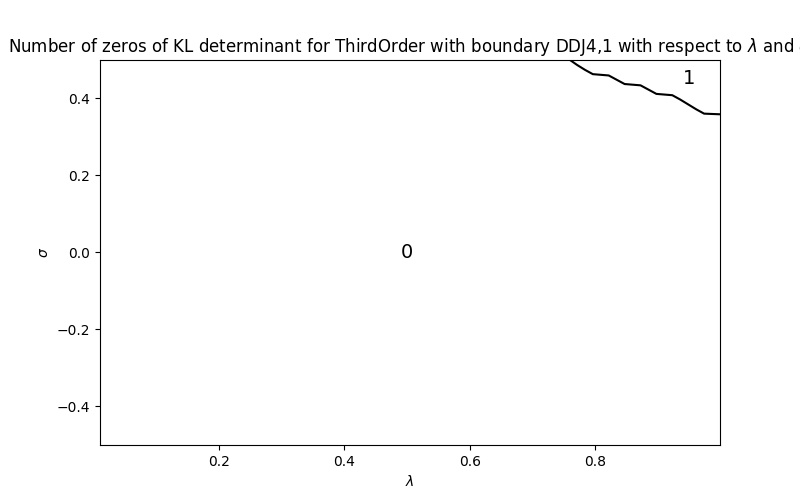}};
        
        \node (s) at (-3.75,2.4) {\scriptsize{$\sigma$}};
        \node (l) at (3.6,-2.4) {\scriptsize{$\lambda$}};
    \end{tikzpicture}

\end{tabular}
\end{center}
\caption{Number of zeros of the Kreiss-Lopatinskii determinant of O3 scheme with different reconstruction boundaries for $\lambda \in ]0,1]$ and $\sigma \in ]-0.5,0.5[$.\label{fig2:asDDJ}}
\end{figure}
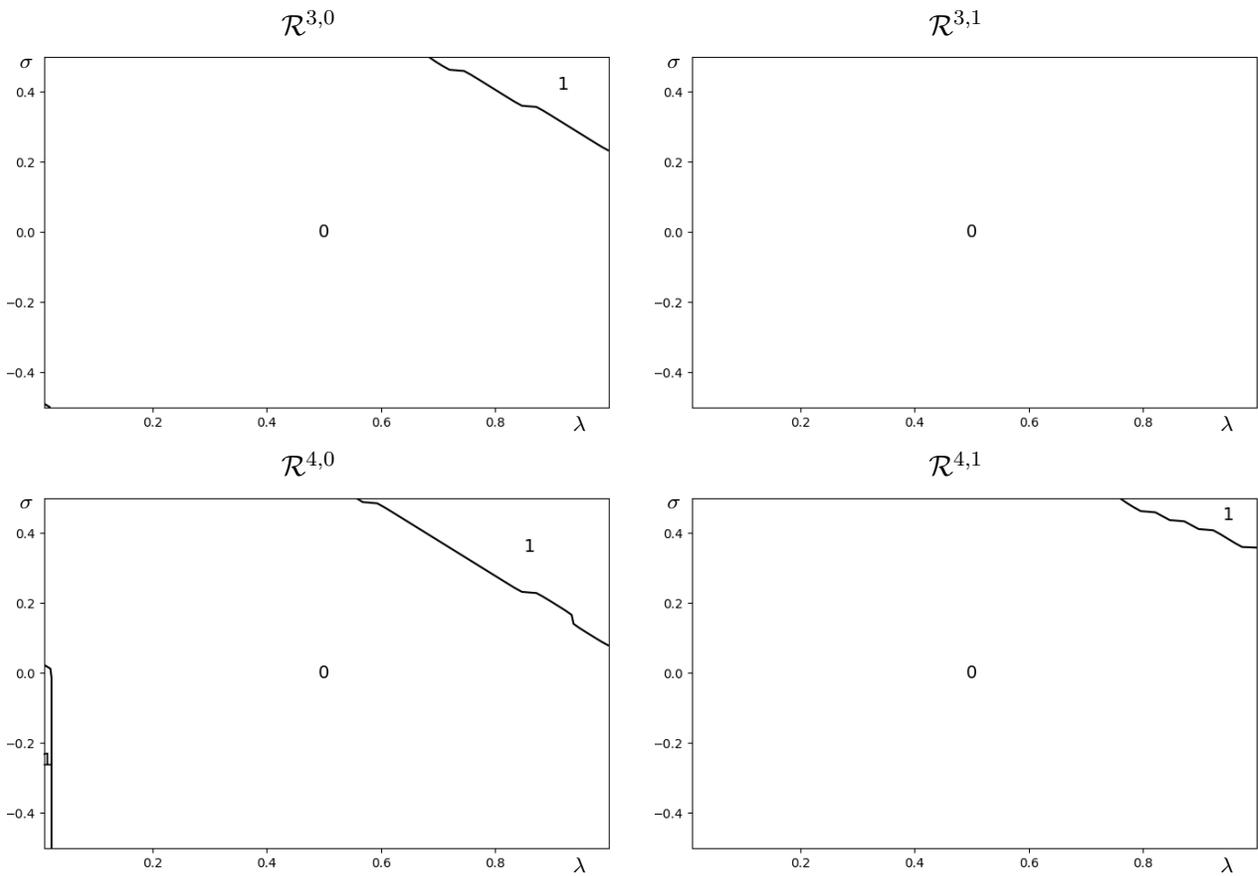


\subsection{Example of Lax-Wendroff 5}

The fifth-order Lax-Wendroff scheme, named LW5, has been proposed in~\cite{Lorcher06} and can be written, for all $n \in \N$, for all $j\in \N$, as
\begin{multline}
    U_j^{n+1} = \tfrac{\lambda(\lambda-2)(\lambda-1)(\lambda+1)(\lambda+2)}{120}U_{j-3}^n - \tfrac{\lambda (\lambda-1)(\lambda-3)(\lambda+1)(\lambda+2)}{24}U_{j-2}^n + \tfrac{\lambda(\lambda-2)(\lambda-3)(\lambda+1)(\lambda+2)}{12}U_{j-1}^n \\
    + \left (1-\tfrac{\lambda(\lambda^4 - 3\lambda^3 - 5\lambda^2 + 15\lambda + 4)}{12} \right ) U_j^n + \tfrac{\lambda(\lambda-1)(\lambda-2)(\lambda-3)(\lambda+2)}{24}U_{j+1}^n - \tfrac{\lambda(\lambda-1)(\lambda-2)(\lambda-3)(\lambda+1)}{120} U_{j+2}^n.
\end{multline}
This scheme LW5 is Cauchy-stable for $\lambda \in ]0,1]$.

\medskip

Figure~\ref{fig2:LW5} illustrates the computation of the number of instabilities for LW5 for different reconstruction boundaries where $\sigma = 0.4$ with respect to $\lambda\in ]0,1]$.
As in the previous example, it may happen that the Kreiss-Lopatinskii curve is too close to the origin, the winding number of the origin cannot be computed correctly. Following the geometric algorithm proposed by Garc\'ia Zapata in~\cite{Zapata14}, a refinement of the discretization then improves the effective computation of the winding number. Figure~\ref{fig2:zoomraffinement} represents such refinement with close-up close to the origin. However, even with this strategy, for very small values of $\lambda$, we cannot refine more than the machine's precision,
that is why there is still some odd pattern for very small $\lambda$ in Figure~\ref{fig2:asDDJ}, Figure \ref{fig2:LW5} and Figure~\ref{fig2:LW5sigma}. As we already discussed, such very small time step are however not used in practice.
The stability area with respect to both parameters $\lambda$ and $\sigma$ are drawn in Figure~\ref{fig2:LW5sigma}, considering again successively various reconstruction boundary conditions. 


\begin{figure}
    \centering
    \begin{tikzpicture}
        \node (A) at (0,0) {\includegraphics[trim = 3.55cm 1.15cm 2.75cm 1.75cm, clip, width = 13cm]{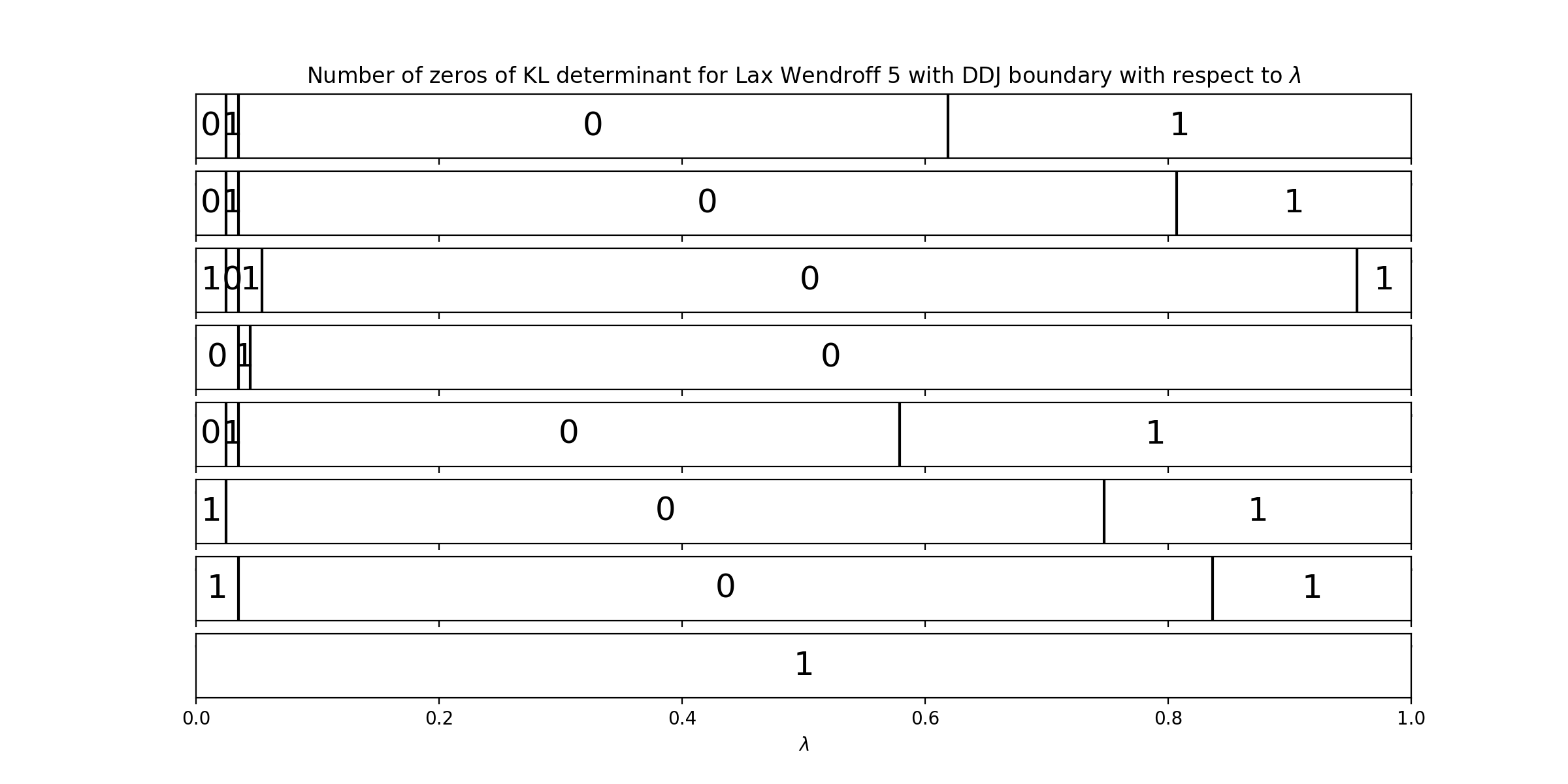}};
        \node (S3ILW4) at (-6.8,2.9) {\scriptsize{$\mathcal R^{5,0}$}};
        \node (S1ILW2) at (-6.8,2.1) {\scriptsize{$\mathcal R^{5,1}$}};
        \node (S2ILW3) at (-6.8,1.3) {\scriptsize{$\mathcal R^{5,2}$}};
        \node (S1ILW3) at (-6.8,0.5) {\scriptsize{$\mathcal R^{5,3}$}};
        \node (S1ILW4) at (-6.8,-0.3) {\scriptsize{$\mathcal R^{6,0}$}};
        \node (S2ILW4) at (-6.8,-1.15) {\scriptsize{$\mathcal R^{6,1}$}};
        \node (S3ILW4) at (-6.8,-1.9) {\scriptsize{$\mathcal R^{6,2}$}};
        \node (S2ILW4) at (-6.8,-2.7) {\scriptsize{$\mathcal R^{6,3}$}};
        \node (l) at (0,-3.4) {\scriptsize{$\lambda$}};
    \end{tikzpicture}

    \caption{Number of zeros of the Kreiss-Lopatinskii determinant of LW5 scheme with different reconstruction boundaries for $\lambda \in ]0,1]$ and $\sigma=0.4$.\label{fig2:LW5}}
\end{figure}
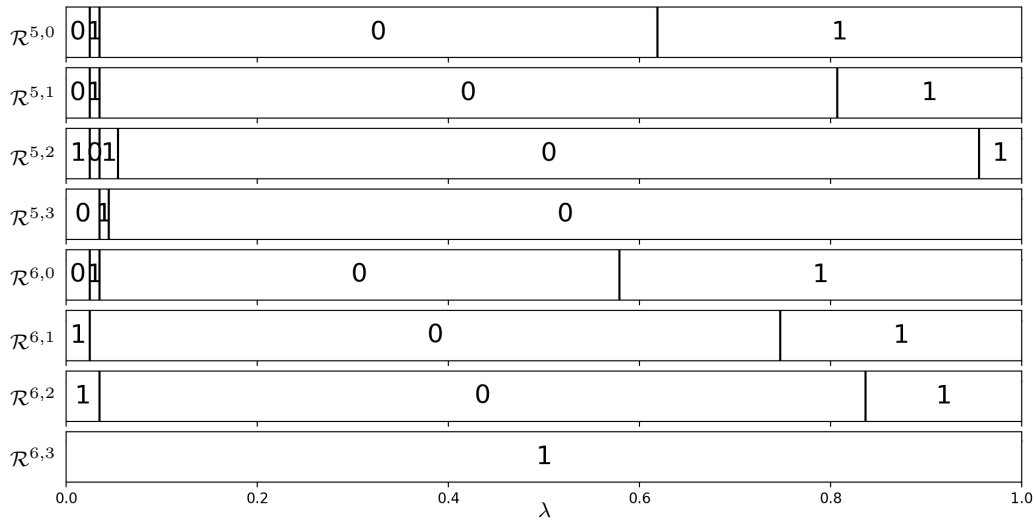


\begin{figure}
    \begin{center}
    \begin{tabular}{cc}
        \begin{tikzpicture}
            \node (A) at (0,0) {\includegraphics[trim = 1.2cm 0.5cm 1.2cm 1.2cm, clip, width = 7cm]{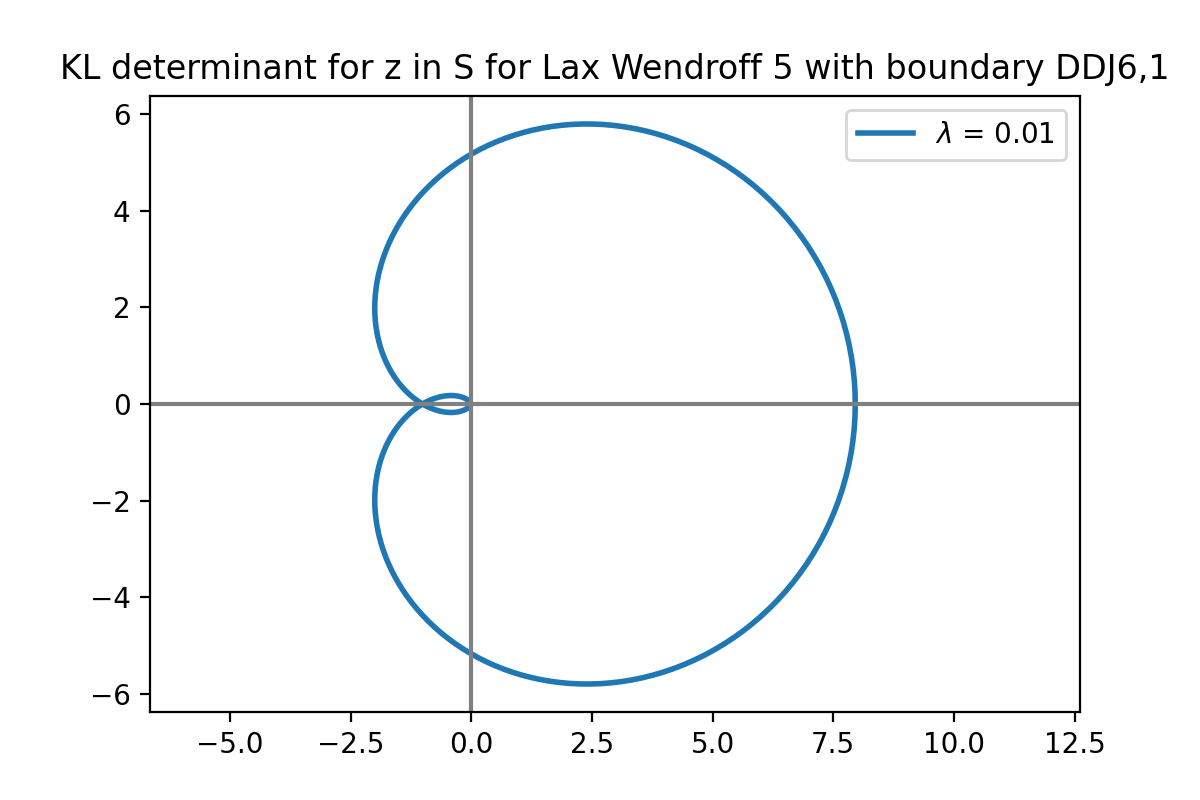}};
            \node (l) at (0,-2.9) {\footnotesize{$(a)$}};
            
        \end{tikzpicture} 
        & \begin{tikzpicture}
            \node (A) at (0,0) {\includegraphics[trim = 1cm 0.5cm 1.2cm 0.8cm, clip, width = 7cm]{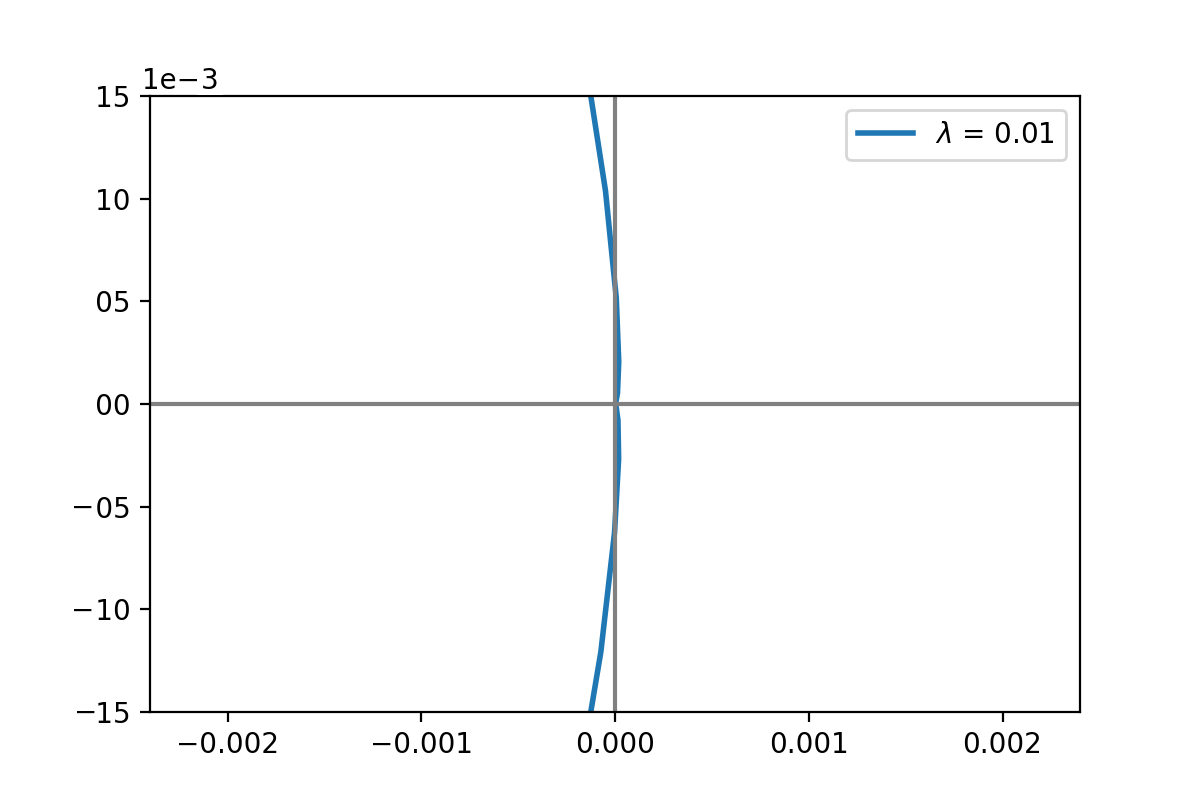}};
            \node (l) at (0,-2.9) {\footnotesize{$(b)$}};
            
        \end{tikzpicture}\\
        \begin{tikzpicture}
            \node (A) at (0,0) {\includegraphics[trim = 1.2cm 0.1cm 1.2cm 0.8cm, clip, width = 7cm]{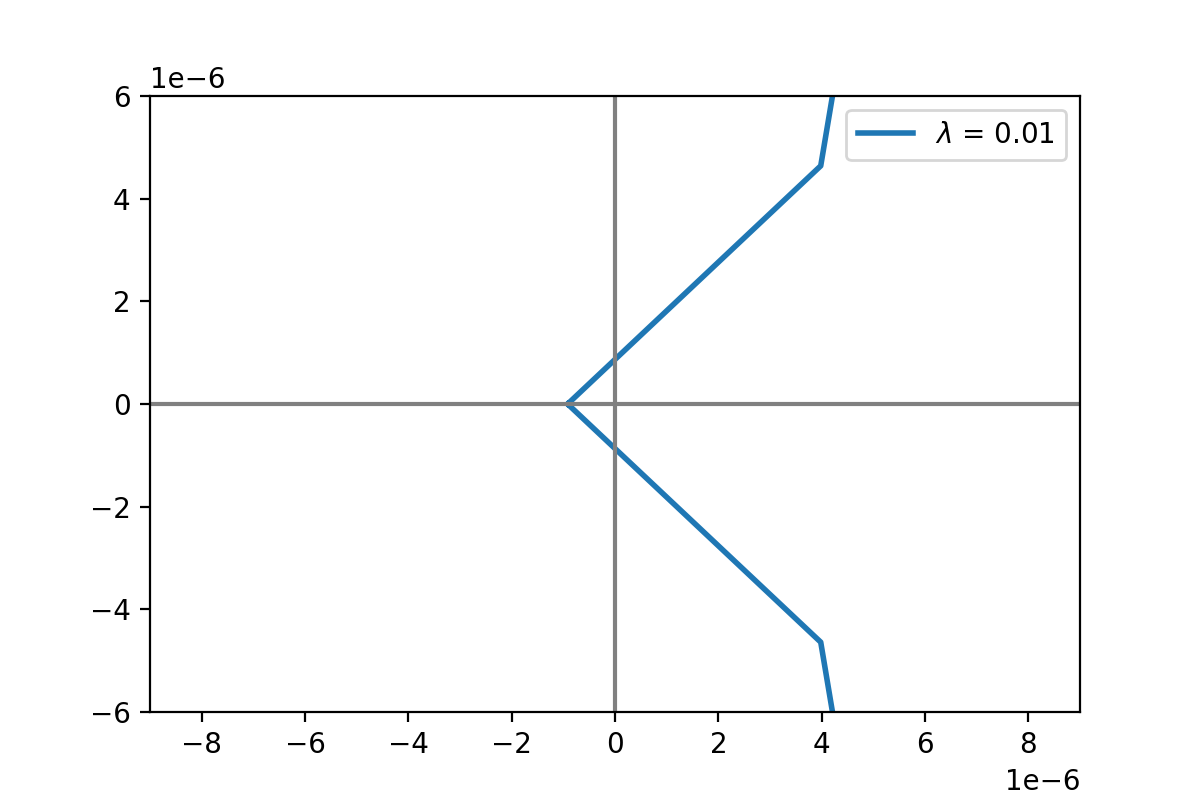}};
            
            \node (l) at (0,-2.9) {\footnotesize{$(c)$}};
        \end{tikzpicture} 
        & \begin{tikzpicture}
            \node (A) at (0,0) {\includegraphics[trim = 1.2cm 0.1cm 1.2cm 0.8cm, clip, width = 7cm]{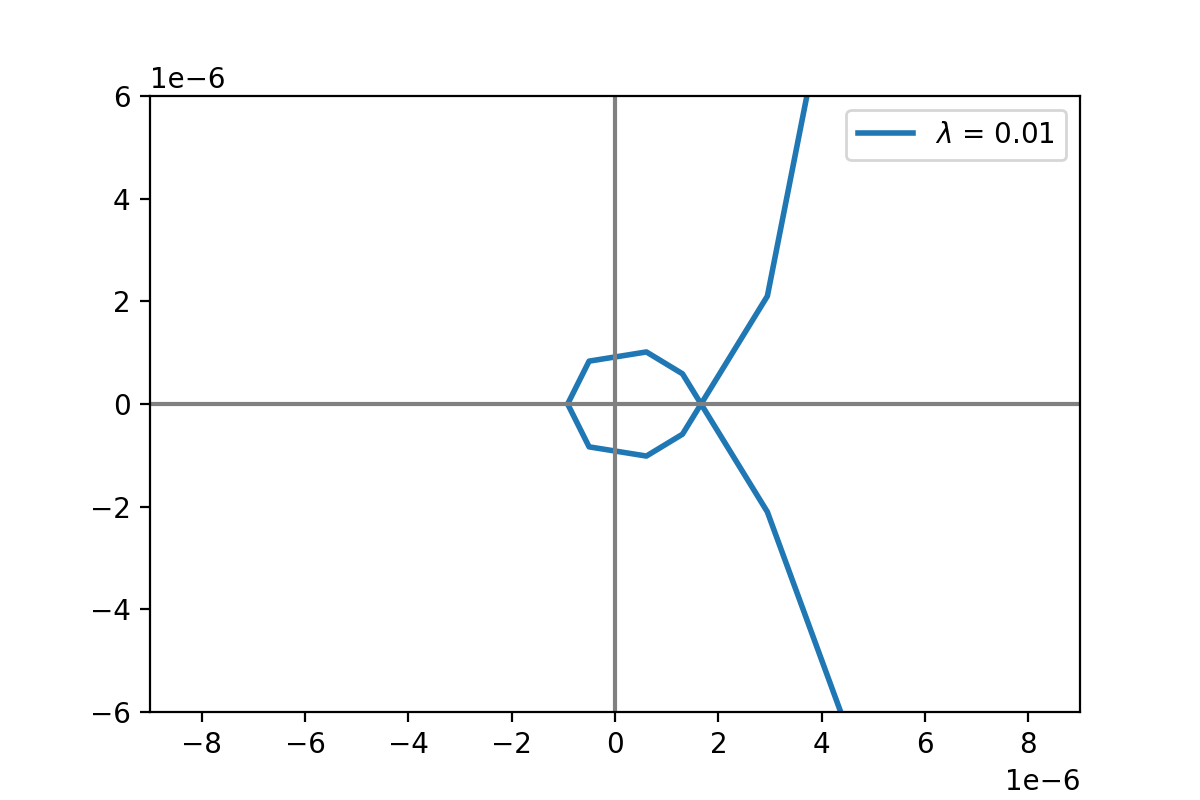}};
            
            \node (l) at (0,-2.9) {\footnotesize{$(d)$}};
        \end{tikzpicture}
    \end{tabular}
    \end{center}
    \caption{$(a)$ representation of $\DKLindep(\S)$ for LW5 with the boundary condition $\mathcal R^{6,1}$, for $\lambda = 0.01$ and $\sigma = 0$, zoom $(b)$ without refinement $(c)$ and with refinement $(d)$.\label{fig2:zoomraffinement}}
    \end{figure}
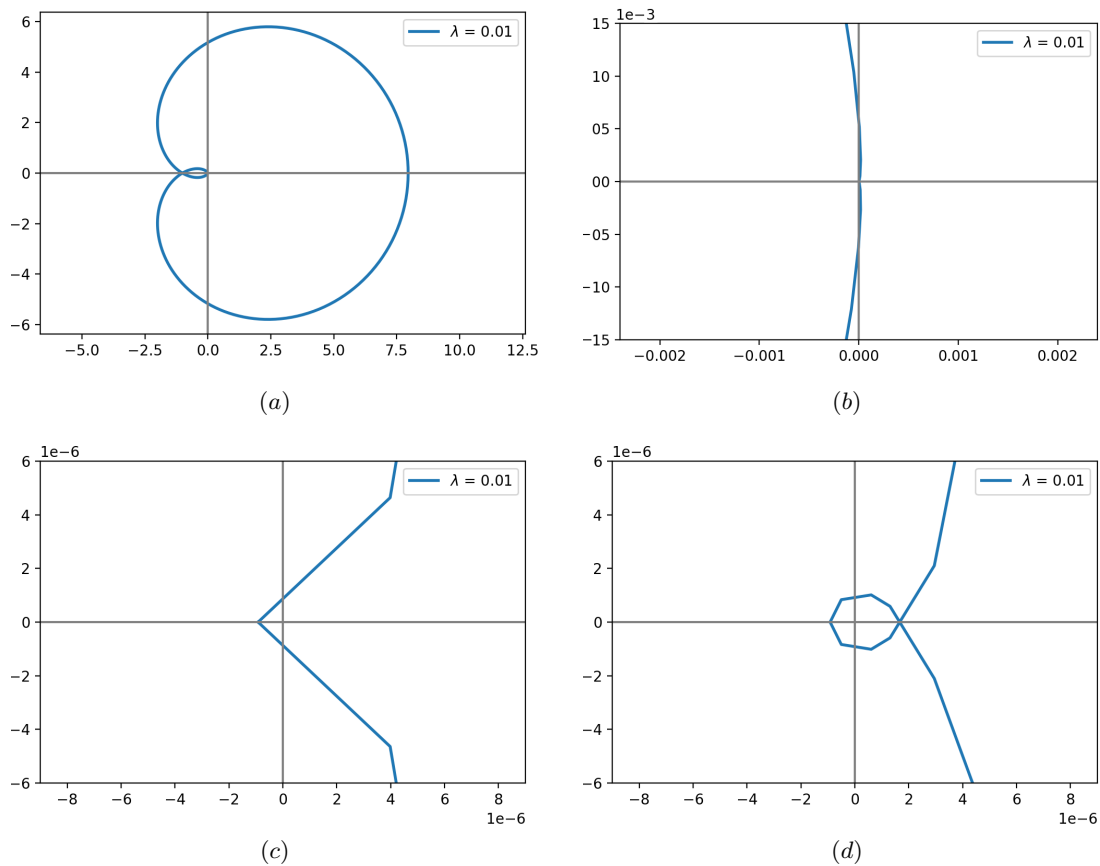

\begin{figure}
    \begin{center}
    \begin{tabular}{cc}
        \begin{tikzpicture}
            \node (A) at (0,0) {\includegraphics[trim = 1.5cm 0.7cm 2cm 1.45cm, clip, width = 8cm]{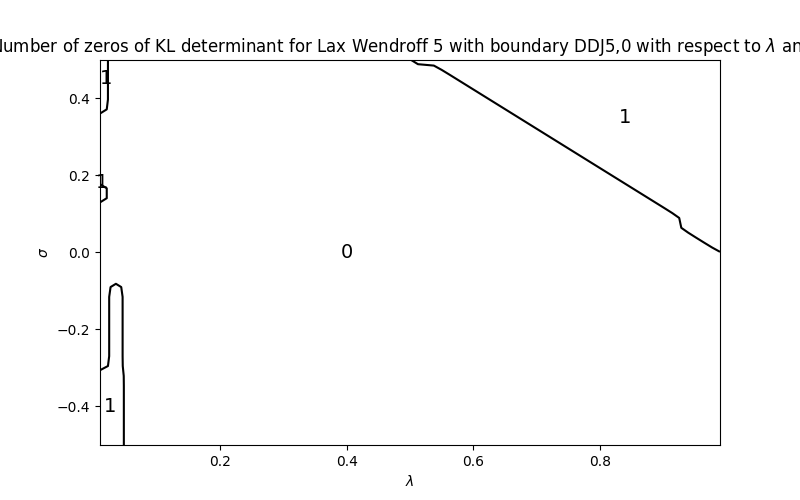}};
            
            \node (s) at (-3.75,2.4) {\scriptsize{$\sigma$}};
            \node (l) at (3.6,-2.4) {\scriptsize{$\lambda$}};
            \node (l) at (3.45,-1.7) {$\mathcal R^{5,0}$};
        \end{tikzpicture} & \begin{tikzpicture}
            \node (A) at (0,0) {\includegraphics[trim = 1.5cm 0.7cm 2cm 1.45cm, clip, width = 8cm]{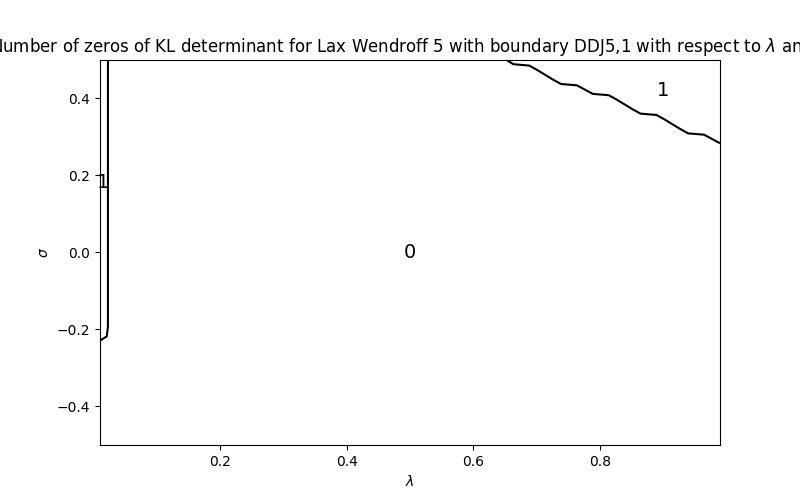}};
            
            \node (s) at (-3.75,2.4) {\scriptsize{$\sigma$}};
            \node (l) at (3.6,-2.4) {\scriptsize{$\lambda$}};
            \node (l) at (3.45,-1.7) {$\mathcal R^{5,1}$};
        \end{tikzpicture} \\ 
        \begin{tikzpicture}
            \node (A) at (0,0) {\includegraphics[trim = 1.5cm 0.7cm 2cm 1.45cm, clip, width = 8cm]{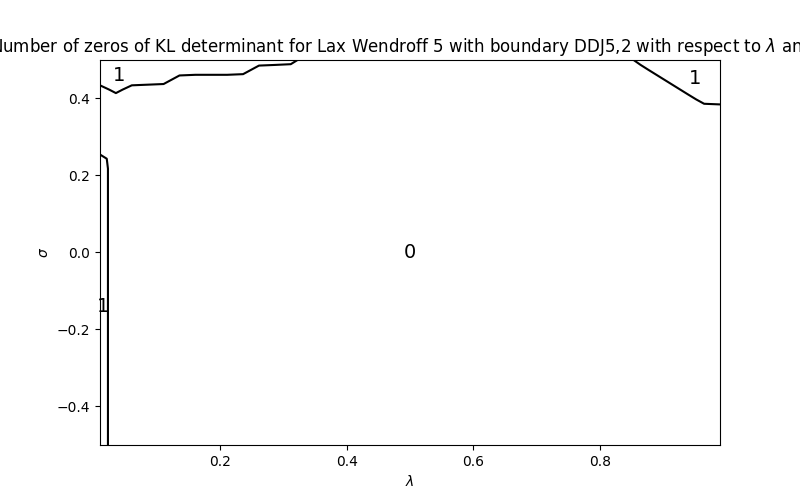}};
            
            \node (s) at (-3.75,2.4) {\scriptsize{$\sigma$}};
            \node (l) at (3.6,-2.4) {\scriptsize{$\lambda$}};
            \node (l) at (3.45,-1.7) {$\mathcal R^{5,2}$};
        \end{tikzpicture} & \begin{tikzpicture}
            \node (A) at (0,0) {\includegraphics[trim = 1.5cm 0.7cm 2cm 1.45cm, clip, width = 8cm]{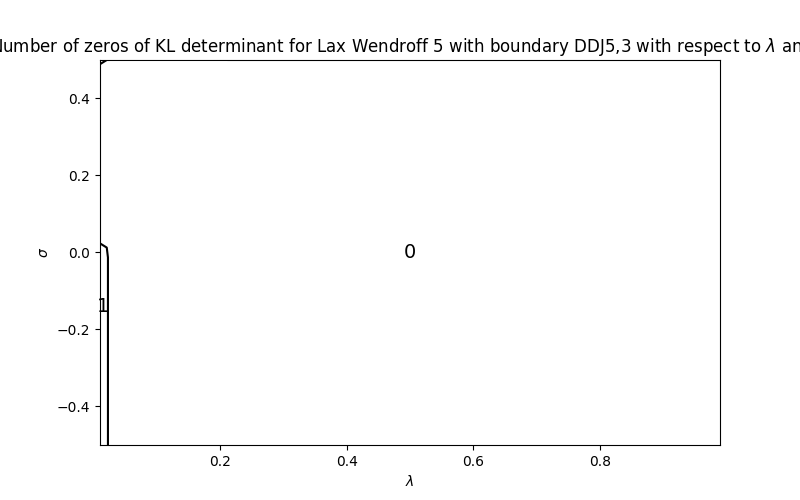}};
            
            \node (s) at (-3.75,2.4) {\scriptsize{$\sigma$}};
            \node (l) at (3.6,-2.4) {\scriptsize{$\lambda$}};
            \node (l) at (3.45,-1.7) {$\mathcal R^{5,3}$};
        \end{tikzpicture} \\ 
        \begin{tikzpicture}
            \node (A) at (0,0) {\includegraphics[trim = 1.5cm 0.7cm 2cm 1.45cm, clip, width = 8cm]{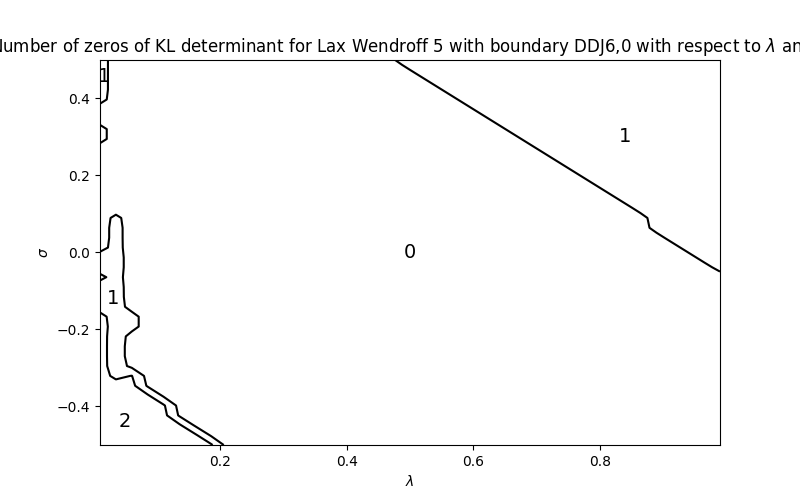}};
            
            \node (s) at (-3.75,2.4) {\scriptsize{$\sigma$}};
            \node (l) at (3.6,-2.4) {\scriptsize{$\lambda$}};
            \node (l) at (3.45,-1.7) {$\mathcal R^{6,0}$};
        \end{tikzpicture} & \begin{tikzpicture}
            \node (A) at (0,0) {\includegraphics[trim = 1.5cm 0.7cm 2cm 1.45cm, clip, width = 8cm]{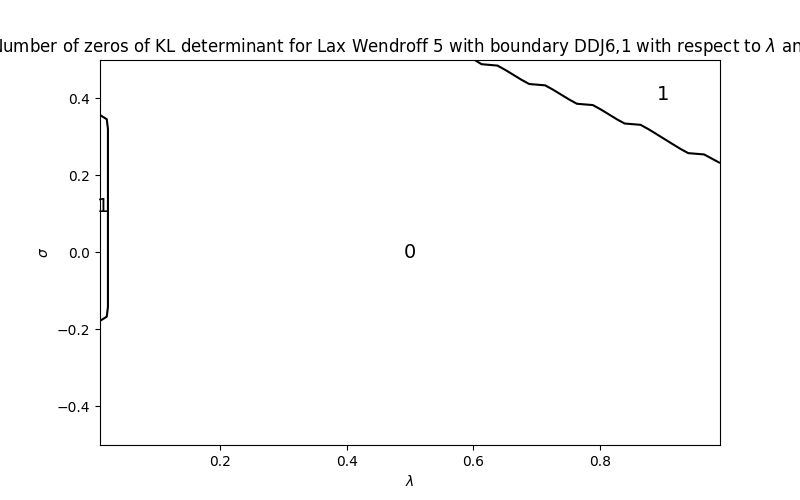}};
            
            \node (s) at (-3.75,2.4) {\scriptsize{$\sigma$}};
            \node (l) at (3.6,-2.4) {\scriptsize{$\lambda$}};
            \node (l) at (3.45,-1.7) {$\mathcal R^{6,1}$};
        \end{tikzpicture} \\ 
        \begin{tikzpicture}
            \node (A) at (0,0) {\includegraphics[trim = 1.5cm 0.7cm 2cm 1.45cm, clip, width = 8cm]{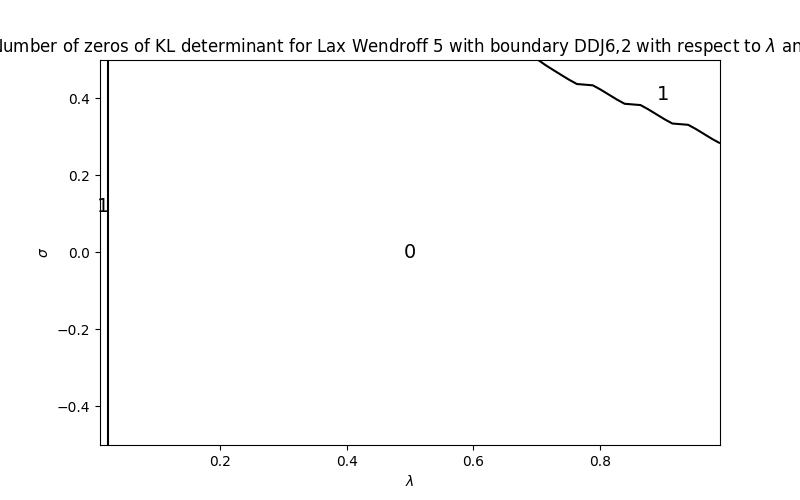}};
            
            \node (s) at (-3.75,2.4) {\scriptsize{$\sigma$}};
            \node (l) at (3.6,-2.4) {\scriptsize{$\lambda$}};
            \node (l) at (3.45,-1.7) {$\mathcal R^{6,2}$};
        \end{tikzpicture} & \begin{tikzpicture}
            \node (A) at (0,0) {\includegraphics[trim = 1.5cm 0.7cm 2cm 1.45cm, clip, width = 8cm]{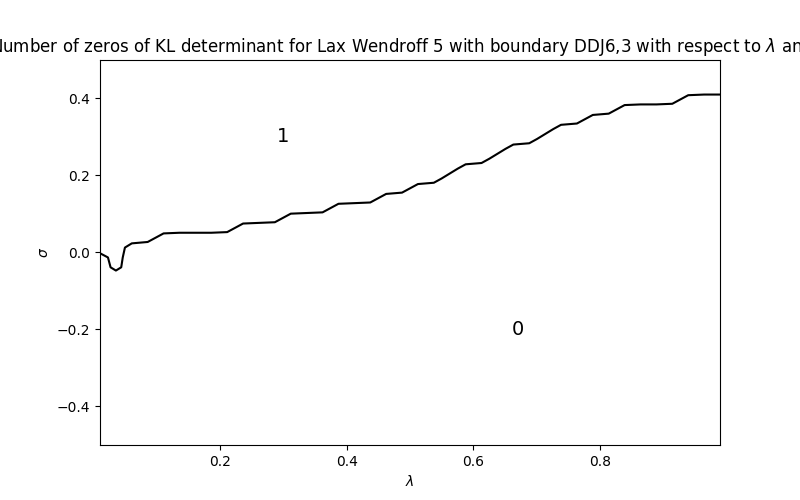}};
            
            \node (s) at (-3.75,2.4) {\scriptsize{$\sigma$}};
            \node (l) at (3.6,-2.4) {\scriptsize{$\lambda$}};
            \node (l) at (3.45,-1.7) {$\mathcal R^{6,3}$};
        \end{tikzpicture}
    \end{tabular}
    \end{center}
    \caption{Number of zeros of the Kreiss-Lopatinskii determinant of LW5 scheme with different reconstruction boundaries for $\lambda \in ]0,1]$ and $\sigma \in ]-0.5,0.5[$.\label{fig2:LW5sigma}}
    \end{figure}

    \bigskip

All the figures can be easily computed in Python with the common  NumPy~\cite{Numpy20} library and the SymPy~\cite{Sympy17} library for the computer algebra system. The algorithm is really efficient. For each subfigure of Figure~\ref{fig2:asDDJ}, the 1600 runs takes less than a couple of minutes of computation achieved on a standard laptop. Moreover, our procedure provides sharp results. In particular, contrary to numerical investigations of stability, which are based on the computation of the spectral radius, no arbitrary truncation of (quasi-)Toeplitz matrices is needed.

\newpage
\bibliographystyle{abbrv}
\bibliography{biblio}


\end{document}